\documentclass[12pt]{article}

\usepackage{amsmath,amssymb, comment}

\usepackage{graphicx}
\usepackage{epstopdf}
\usepackage{hhline}
\usepackage{color}
\usepackage{epsfig}

\newcommand{\bu}{{\mathbf u}}

\newcommand{\R}{{\Bbb R}}

\newcommand{\what}{\widehat}
\newcommand{\wtilde}{\widetilde}

\newtheorem{definition}{Definition}

\newtheorem{theorem}{Theorem}
\newtheorem{lemma}{Lemma}

\newenvironment{proof}{\begin{list}{$\!\!${\bf Proof} \rule{1pt}{0pt}}
{\setlength{\leftmargin}{0pt}\setlength{\itemindent}{30pt}\setlength{\listparindent}
{15pt}}\item}{\rule{0.3em}{0mm}\hfill\framebox[1.2ex]{\rule{0.3em}{0mm}}
\end{list}}
\title{Preconditioning for Accurate Solutions of Linear Systems and Eigenvalue Problems
}

\author{Qiang Ye\thanks{Department of Mathematics, University of Kentucky,
Lexington, KY 40506. {\tt qye3@uky.edu}. Research supported in part by NSF Grants DMS-1317424, DMS-1318633 and DMS-1620082.
}
}
\date{}

\begin{document}

\maketitle

\begin{abstract}
This paper develops the preconditioning technique as a method to address the accuracy issue caused by ill-conditioning. Given a preconditioner $M$ for an ill-conditioned linear system $Ax=b$, we show that, if  the inverse of the preconditioner $M^{-1}$ can be applied to vectors {\em accurately}, then the linear system can be solved {\em accurately}. A stability concept called {\em inverse-equivalent} accuracy is introduced to describe higher accuracy that is achieved and an error analysis will be presented. As an application,  we  use the preconditioning approach to accurately compute a few smallest eigenvalues of certain ill-conditioned matrices. Numerical examples are presented to illustrate the error analysis and the performance of the methods.
\end{abstract}

%
%\begin{keywords}
%Eigenvalue; Ill-conditioned matrix; Lancozos Method; differential eigenvalue problem, biharmonic operator.
%\end{keywords}
%
%\begin{AMS}
%65F15, 65F35, 65N06, 65N25
%\end{AMS}

\pagestyle{myheadings}
\thispagestyle{plain}

\section{Introduction}\label{sec:intro}

Solutions of large scale linear algebra problems are   typically associated with an ill-conditioned  matrix $A$ where the condition number  $\kappa (A):=\|A\|\|A^{-1}\|$ is large.
The ill-conditioning has two effects in numerically solving a linear system $Ax=b$. It reduces the rate of convergence of iterative algorithms such as the Krylov subspace methods. It also limits the accuracy to which $Ax=b$ can be solved in finite precision.
The former problem is typically addressed by a technique known as preconditioning. % where we solve $M^{-1}Ax = M^{-1}b$ for some $M\approx A$ that can be easily inverted.
For the latter, there is no known good solution other than the classical diagonal scaling or iterative refinements; see   \cite[Sec. 2.5]{demmel97} and \cite[p.124]{gova}.

While a large condition number $\kappa (A)$ is typically associated with the two difficulties discussed above in solving linear systems, it also causes  two similar problems for eigenvalue computations. First,  a large $\kappa (A)$ is often  associated with a spectrum that has one or both ends clustered, %i.e. their gaps are small relative to a large spread of spectrum; see (\ref{eq:relgap}). This
which results in slow convergence for methods such as the Lanczos/Arnoldi algorithms.
A large $\kappa (A)$ also limits the accuracy of those smaller eigenvalues of $A$ computed in finite precision; see %\cite[Sec. 5.2.1]{demmel97}
\cite{demmel99b,demmelkahan} or  \S \ref{sec:eigen} for some discussions.
%For example, for a general symmetric eigenvalue problem, the smallest eigenvalue computed by a standard  algorithm has a relative accuracy proportional to the condition number; see \cite{demmelkahan,demmel99b} and \S \ref{sec:eig} for more detailed discussions.
The shift-and-invert transformation and its  variants  are efficient ways of dealing with clustering;
%which include the  Jacobi-Davidson type methods,  LOBPCG, and our own inverse-free preconditioned Krylov subspace methods.
see \cite{bai} for example. The relative accuracy issue has also been studied extensively and several algorithms have been developed for various structured matrices
for which all singular values or eigenvalues can be computed to an  accuracy independent of the condition number; see   \cite{axy2,barlowdemmel,demmel99a,demmel99b,demmelkahan,demmelkoev,deve,doko,doko11,Dopico:09,drma14,fepa,li98,ye08} and the references contained therein.

The preconditioning technique is a  general methodology that has been highly successful in overcoming the effect of ill-conditioning on the speed of convergence of  iterative methods for solving a linear system $Ax=b$. Given an  invertible $M\approx A$, we implicitly transform  the linear system to the well-conditioned one,  $M^{-1}Ax=M^{-1}b$,    which can be solved iteratively with accelerated convergence. This poses the natural question: Do we also obtain a more accurate solution by solving the preconditioned  system $M^{-1}Ax=M^{-1}b$? The answer is generally no. Because $M$ is a good preconditioner to an ill-conditioned $A$, it is necessarily ill-conditioned and hence there are potentially large roundoff errors encountered in forming the preconditioned system either explicitly or implicitly; see \S \ref{sec:inv}  and \S \ref{sec:example} for more details and examples. On the other hand, if $M^{-1}Ax=M^{-1}b$ can be formed exactly or sufficiently accurately, solving that will clearly give an accurate solution.  Indeed, diagonal scaling is one such example  where $M$ is chosen to be a diagonal matrix of powers of $2$ so that no roundoff error is generated when applying $M^{-1}$. Thus, the  goal of this paper is to investigate to what accuracy inverting $M$ in preconditioning can lead to improved solution accuracy.

We will develop the preconditioning technique as a method to solve the accuracy issue caused by ill-conditioning for both linear systems and eigenvalue problems. We will show that preconditioning can indeed lead to highly satisfactory solution accuracy of a linear system if the inverse of the preconditioner, $M^{-1}$, can be applied sufficiently {\em accurately}. % and the preconditioned system is formed in a certain way.
To study precisely the accuracy that is needed for $M^{-1}$  and that can be attained by the final solution, we will introduce a stability concept called {\em inverse-equivalent} accuracy, which is one equivalent to multiplying exact inverses.
%Our preconditioning method requires that $M$  be inverted with the {\em inverse-equivalent} accuracy but otherwise there is no condition on $A$ and $M$.
An error analysis together with numerical examples will be presented to demonstrate the stability gained. While the present paper is focused on linear systems,  we will also use this accurate preconditioning method to accurately compute a few smallest eigenvalues of an ill-conditioned matrix through the  accurate inverse approach presented in \cite{ye17}.

We remark that the only requirement for the accurate preconditioning process is that $M$  be inverted with the {\em inverse-equivalent} accuracy. This can be done if $M^{-1}$ is explicitly available or $M$ has an accurate rank-revealing decomposition (see \cite{demmel99b,domo12}). In \cite{demmel99b}, several classes of structured matrices have been shown to have an accurate rank-revealing decomposition, which include graded matrices, total signed compound matrices such as acyclic matrices, Cauchy matrices, totally positive  matrices, diagonally scaled totally unimodular matrices, and matrices arising in certain simple finite element problems. We have also shown  in \cite{ye09} that diagonally dominant matrices have an accurate rank-revealing decomposition. Thus, the accurate preconditioning method is applicable to a broad class of matrices  that can be well preconditioned by any of these structured matrices.

The rest of the paper is organized as follows. We  present in \S \ref{sec:iea} the concept of {\em inverse-accurate} accuracy. We then develop in \S   \ref{sec:inv} the accurate preconditioning method and an error analysis for linear systems. In \S \ref{sec:eigen}, we discuss applying the accurate preconditioning method to accurately compute a few smallest eigenvalues of a matrix. Finally, in \S \ref{sec:example}, we present some numerical examples for both   linear systems  and eigenvalue problems, followed by some concluding remarks in \S \ref{sec:conclusion}.

\subsection{Notation and Preliminaries}\label{sec:pre}
Throughout this paper, $\| \cdot \|$ denotes a general norm for vectors and its induced operator norm for matrices.  $\| \cdot \|_p$ denotes the $p$-norm. Inequalities and absolute value involving matrices and vectors are entrywise.

For error analysis in a floating point arithmetic, $\bu$ denotes the machine roundoff unit and ${\cal O} (\bu)$ denotes a term bounded by $p(n)\bu$ for some polynomial $p(n)$ in $n$. We use $fl(z)$ to denote the computed result of an algebraic expression $z$. We  assume throughout that matrices and vectors given have floating point number entries.
We assume the following standard model for roundoff
errors in  basic matrix computations \cite[p.66]{gova}:
\begin{equation}
fl(x+y) =  x + y + e \quad \mbox{ with } \quad
|e| \le {\bf u} (| x + y |)
\label{eq:xpy}
\end{equation}
\begin{equation}
fl(A x) = A x +  e \quad \mbox{ with } \quad
|e| \le {\bf u} N |A| |x|  + {\cal O}({\bf u}^2),
\label{eq:dot}
\end{equation}
%where $x, y$ are floating point vectors.
where $N$ is the maximal number of nonzero entries per row of $A$.
Using (2.4.12) of \cite[p.64]{gova} and equivalence of any two norms in a finite dimensional space, we may also simply rewrite (\ref{eq:dot}) as
\begin{equation}
\|fl(A x) - A x \| \le {\cal O} (\bu) N \|A\| \|x\|.
\label{eq:dot2}
\end{equation}
This bound is based on explicitly multiplying $A$ with $x$ and $N\le n$  can be absorbed into the ${\cal O} (\bu)$ term.  More generally, if $A$ is not explicitly given and $Ax$ is computed as an operator,  (\ref{eq:dot2}) may still be valid if we allow $N$ to be a suitable constant associated with the operator $Ax$.

\section{Inverse-equivalent  Accuracy}\label{sec:iea}
In this section, we introduce a stability concept called {\em inverse-equivalent} accuracy for solving linear systems in finite precision.

Given an invertible matrix $A\in \R^{n\times n}$  and $b \in R^{n}$, all standard dense algorithms for solving the linear system $Ax=b$ in a floating point arithmetic computes a solution $\hat x$ that is backward stable, i.e. it satisfies $(A+E) \hat x = b$ for some $E$ with $\|E\|/\|A\| = {\cal O} (\bu)$.
%Recall that $\bu$ is the machine roundoff unit and
%${\cal O} (\bu)$ denotes a term bounded  by $p(n)\bu$ for some polynomial $p(n)$.
An iterative method computes a solution $\hat x$ with a residual that at best satisfies $\|b-A \hat x \|={\cal O} (\bu) \|A\| \|\hat x\|$, which is equivalent to the backward stability.  In both cases, the solution error is bounded as
\begin{equation}\label{eq:cond}
\frac{\|\hat x -x\|}{\|x\|} \le {\cal O} (\bu) \kappa (A), \;\;\mbox{ where }\;\; \kappa (A)= \|A\|  \|A^{-1}\|.
\end{equation}

This backward stable solution accuracy may be unsatisfactory for ill-conditioned problems, but for a general linear system, this is the best one may hope for because the solution is not well determined by the matrix $A$ under perturbations. For many ill-conditioned linear systems arising in applications, however, the underlying solution  may be much more stable when considering the solution as determined from the original problem data rather than from the matrix. For example,  discretization of a differential equation typically gives rise to an ill-conditioned linear system, but its solution, approximating the solution of PDE, is stably  determined by the input data of the PDE. Namely, the solution is  stable if we only consider perturbations to the problem data in spite of ill-conditioning of the matrix. In that case, we are interested in special algorithms that can solve such ill-conditioned linear systems  more accurately.
% if all the entries of $A$ are allowed to have some uncertainties/numerical errors. However, for a large sparse matrix arising in applications, uncertainties or measurement errors in the problem data do not yield a perturbation to the entire matrix, but only those nonzero entries of $A$ and $b$ that are defined by the data. Hence, its solution may be more accurately determined from the problem data than what (\ref{eq:cond}) suggests, which assume a general perturbation $E$.

Before we study algorithms, we first address how high an accuracy  one may reasonably expect to achieve for a linear system. Ideally, we may strive for the full relative accuracy
\begin{equation}\label{eq:ideal}
\frac{\|\hat x -x\|}{\|x\|} \le {\cal O} (\bu)
\end{equation}
but a bound totally independent of $A$ will obviously require very stringent conditions on $A$, as a perturbation to $b$ alone will produce errors proportional to  $A^{-1}$. Note that $b$ typically corresponds to problem data and then some  perturbations/uncertainties in $b$ should be assumed. Furthermore, the ideal accuracy (\ref{eq:ideal}) may not be necessary in many applications.
%Note that it is unlikely to have practical problems with $b$ independent of the original data of the problem.
%Thus, we will show that the accuracy defined below may be achieved by the preconditioning method, which has also been considered by Chan and Foulser \cite{chan:88}.
Indeed, the accuracy we introduce now is often sufficient in applications.

\begin{definition}\label{def:1}
%We say that a computed solution $\hat x$ to $Ax=b$ in the floating point arithmetic has inverse-equivalent accuracy if
Given $A$, we say that an algorithm for solving linear systems with coefficient $A$ is inverse-equivalent if,  for any $b$, it  produces in a floating point arithmetic a computed solution $\hat x$  to $Ax=b$   such that
\begin{equation}\label{eq:invacc}
    \|\hat x -x\| \le {\cal O} (\bu)  \|A^{-1}\| \|b\|.
\end{equation}
We also say such a solution $\hat x$  has an inverse-equivalent accuracy.
\end{definition}

In the definition, we have used a general norm. Since any two norms are equivalent and ${\cal O} (\bu)$ can absorb any   constant, the definition is equivalent to one using any particular norm in (\ref{eq:invacc}). The next two results explain the naming of this accuracy.
\begin{theorem}\label{thm:invacc}
If $A$ is such that $A^{-1}$ is explicitly available, then solving $Ax=b$ by
%the solution $\hat x = fl(A^{-1} b)$ obtained by
multiplying $A^{-1}  $ with $b$ is an inverse-equivalent algorithm.
\end{theorem}
\begin{proof}
Recall that $A$ and $b$ are assumed to have floating point number entries. For $A^{-1}$, we have $|fl(A^{-1}) - A^{-1}| \le \bu |A^{-1}|$. Then $\|fl(A^{-1})b - A^{-1}b\| \le {\cal O} (\bu) \|A^{-1}\| \|b\|$. It follows from (\ref{eq:dot}) that  $\|fl(A^{-1} b) - fl(A^{-1}) b\| \le  {\cal O} (\bu) \|fl(A^{-1})\| \|b\|.$ Combining the two, we obtain  $\|fl(A^{-1} b) - A^{-1} b\| \le {\cal O} (\bu)  \|A^{-1}\| \|b\|.$
\end{proof}

\begin{theorem}\label{thm:invaccA}
Let  $A$ be an invertible matrix. There is an inverse-equivalent algorithm for $A$
%solution of $Ax=b$ can be computed by some algorithm  for any $b$
if and only if  the inverse $A^{-1}$  can be computed by some algorithm with a relative error of order ${\cal O}(\bu)$, i.e.  the computed inverse  $\what X$  %by solving $AX=I$ using the algorithm
satisfies
\begin{equation}\label{eq:invA}
\frac{\|\what X -A^{-1}\|}{\|A^{-1}\|} \le {\cal O} (\bu)
\end{equation}
\end{theorem}
\begin{proof}
First assume that there is an inverse-equivalent algorithm for $A$. Using this algorithm to compute the inverse $A^{-1}$ by solving $AX=I$, let the computed inverse be $\what X =[\hat x_1, \hat x_2, \cdots, \hat x_n]$ and write $X=A^{-1}=[ x_1,  x_2, \cdots,  x_n]$. Then $\hat x_i$ is inverse-equivalent, i.e., written in the 1-norm, $\|\hat x_i - x_i \|_1 \le {\cal O} (\bu) \|A^{-1}\|_1 \|e_i\|_1 = {\cal O} (\bu) \|A^{-1}\|_1$. Thus $\|\what X - X\|_1 = \max_i \|\hat x_i - x_i \|_1 \le {\cal O} (\bu) \|A^{-1}\|_1 $. By equivalence of norms, (\ref{eq:invA}) is proved.

On the other hand, if we have an algorithm that computes the inverse $\what X$ satisfies  (\ref{eq:invA}), then for any $b$, solving $Ax=b$ by computing $\hat x = fl(\what X b)$, we have
\begin{eqnarray*}
\|\hat x - x\| % &=& \|fl(\what X b) - \what X b + \what X b -A^{-1} b\| \\
&\le& \|fl(\what X b) - \what X b \|+ \|\what X b -A^{-1} b\| \\
&\le&  {\cal O} (\bu) \|\what X\| \|b\|   +   {\cal O} (\bu) \|A^{-1}\| \| b\| \\
&\le&  {\cal O} (\bu) ( \|\what X-A^{-1} \| +\|A^{-1} \|) \|b\|   +   {\cal O} (\bu) \|A^{-1}\| \| b\| \\
&\le&  {\cal O} (\bu) ( {\cal O} (\bu) \|A^{-1}\|  +\|A^{-1} \|) \|b\|   +   {\cal O} (\bu) \|A^{-1}\| \| b\| \\
 &=&  {\cal O} (\bu) \|A^{-1}\| \| b\|.
\end{eqnarray*}
So the algorithm  $\hat x = fl(\what X b)$ is inverse-equivalent. This completes the proof.
\end{proof}

The above shows that an inverse-equivalent algorithm produces solution that are comparable to the one obtained by multiplying the exact inverse with the right-hand side vector $b$. This should be highly satisfactory in many applications. For example, in eigenvalue computations with the shift-and-invert transformation, using an inverse-equivalent algorithm for the inverse would produce results as accurate as the one obtained using the exact inverse; see  \S \ref{sec:eigen}.

If we rewrite (\ref{eq:invacc}) in the relative error form
\begin{equation}\label{eq:invacc-rel}
    \frac{\|\hat x -x\|}{\|x\|} \le {\cal O} (\bu)  \frac{\|A^{-1}\| \|b\|}{\|x\|},
\end{equation}
then it is clear that this accuracy is between the full relative accuracy  (\ref{eq:ideal}) and the backward stable solution accuracy (\ref{eq:cond}) as $ \|x\| \le \|A^{-1}\| \|b\| \le \|A^{-1}\| \|A\| \|x\|$. Note that the bound (\ref{eq:invacc-rel}) has also appeared in the study of perturbation theory for $Ax=b$ when only the right-hand side vector $b$ is perturbed; see \cite{chan88,high02}.
%Indeed, $\kappa (A, b):= \frac{\|A^{-1}\| \|b\|}{\|x\|}$ is  the associated condition number.
It has been observed that the bound  (\ref{eq:invacc-rel})  may be substantially smaller than  (\ref{eq:cond}); see \cite{chan88,domo12}. For example, this occurs as long as $b$ has a significant projection on some right singular vector $u_k$ of $A$ corresponding to a singular value $\sigma_k$ that is far less than the largest one. Namely, if $\sigma_1 \ge \sigma_2 \ge \ldots \ge \sigma_n$ are the singular values of $A$, then $\|x\|_2=\|A^{-1} b\|_2 \ge |u_k^T b| /\sigma_k$ and hence
\begin{equation}\label{eq:ll}
\frac{\|A^{-1}\|_2 \|b\|_2}{\|x\|_2} \le \frac{\sigma_k}{\sigma_n} \frac{  \|b\|_2}{|u_k^T b|} \ll \|A\|_2  \|A^{-1}\|_2
\end{equation}
if
\begin{equation}\label{eq:bcond}
\frac{\sigma_k }{ \cos \angle (b, u_k)} \ll \sigma_1.
\end{equation}
See \cite{chan88,domo12} for some more detailed discussions.

We remark that  $b$, being the input data in a practical problem,  is unlikely to be nearly orthogonal to all singular vectors corresponding to smaller singular values. For example,  if $b$ is a random vector, (\ref{eq:bcond}) may be easily satisfied. So we may expect the inverse-equivalent accuracy (\ref{eq:invacc-rel}) to be significantly better than the backward stable one  (\ref{eq:cond}) when $b$ is chosen with no   constraint.

\section{Accurate  solutions for linear systems}\label{sec:inv}
In this section, we present an accurate preconditioning method for solving a linear system  where the inversion of the preconditioner is computed with an inverse-equivalent algorithm. We show that this results in inverse-equivalent accuracy and we present our analysis in two subsections, one for direct methods  and one for iterative ones for solving the preconditioned  equation. We first briefly discuss the accuracy that may be expected when  a standard backward stable algorithm is used for the preconditioner.

Preconditioning a linear system $Ax =b$ is commonly used to accelerate convergence of an iterative method. Given a preconditioner $M \approx A$ such that $M^{-1}A$ is well-conditioned, applying an iterative method to $M^{-1}Ax=M^{-1}b$ results in accelerated convergence. Since $M^{-1}Ax=M^{-1}b$  is a well-conditioned system, it might be argued that solving the preconditioned equation
%$M^{-1}Au=M^{-1}v$, with a better conditioned coefficient matrix,
should produce more accurate solutions. However, inverting $M$ encounters roundoff errors which  change the preconditioned system and the final solution. We analyze this error as follows.
%This is unfortunately not the case in general, because the preconditioned equation may not be formed, explicitly or implicitly, sufficiently accurately.

First, we observe that for $M^{-1}A$ to be well-conditioned, $M$ is necessarily  ill-conditioned (i.e. has a   condition number comparable to $A$). This is because
\begin{equation}\label{eq:ma}
 \frac{\kappa (A)}{\kappa (M^{-1}A) }\le \kappa (M) \le \kappa (M^{-1}A)  \kappa (A).
\end{equation}
Then the application of $M^{-1}$ on $A$ and on $b$ can not be computed accurately.
%and the computed preconditioned system $M^{-1}Au=M^{-1}v$ deviate from the exact one significantly.
For example, assuming $M$ is inverted by a backward stable algorithm, the computed result of the right-hand side $M^{-1}b$ is $M^{-1}b +f$ with the error $f$ bounded by $\|f\|/\|M^{-1}b\| = {\cal O}( \bu) \kappa(M)$. Similarly, the  computed result of $M^{-1}A$ is $M^{-1}A +E$ with $\|E\|/\|M^{-1}A\| = {\cal O}( \bu) \kappa(M)$.
%Thus even the exact solution $\widehat u$ of the computed preconditioned system $(M^{-1}A+E) \widehat u = M^{-1}v +f$ will have an error bounded by
Thus, the preconditioned system obtained is
\begin{equation}\label{eq:precon}
(M^{-1}A+E) y = M^{-1}b +f,
%\;\;\mbox{ with }\;\;
%\frac{\|E\|}{\|M^{-1}A\|} = {\cal O}( \bu) \kappa(M),\; \frac{\|f\|}{\|M^{-1}v\|} = {\cal O}( \bu) \kappa(M),
\end{equation}
and then  even its exact solution $y$ can only be bounded as
\begin{equation}\label{eq1}
\frac{\|y - x\|}{\|x\|}
%\le  \kappa (M^{-1}A) \left(\frac{\|E\|}{\|M^{-1}A\|} +\frac{\|f\|}{\|M^{-1}v\|}\right)
\le {\cal O}(\bu) \kappa(M) \kappa (M^{-1}A)
%\le  {\cal O}(\bu) \kappa (A) \kappa^2 (M^{-1}A).
\end{equation}
which  by (\ref{eq:ma}) is approximately ${\cal O}(\bu) \kappa(A)$.
We conclude that the computed solution to  $ M^{-1}Ax=M^{-1}b$, after accounting the errors of inverting $M$, has a relative error of order $\bu \kappa(A)$.
So, the solution accuracy  can not be improved by  preconditioning in general; see numerical examples in \S \ref{sec:example}.

Note that the discussion above is for a general $M$ solved by a backward stable algorithm. The diagonal scaling, where $M$ is chosen to be a diagonal matrix (typically with entries being powers of $2$) \cite[p.124]{gova}, is an effective method for improving solution accuracy provided the diagonal matrix is a good preconditioner. With such a preconditioner, the preconditioning transformation  is  performed exactly and the resulting solution accuracy is indeed improved. This  leads us to the following questions: Can  more accurately inverting $M$  lead to a more accurate solution of the original system, and if so, what accuracy is needed for $M^{-1}$? The rest of this section provides answers to these questions.

%can be obtained from the preconditioned equation if $M^{-1}A$ is well-conditioned and is computed accurately. Here what accuracy is needed is a critical question. An easy answer is to demand  $M^{-1}A$ be computed with a small normwise relative error, but this can not be done for almost all matrices other than diagonal ones.  The rest of this section will show that the inverse-equivalent precision is a reasonable accuracy required of the preconditioner and expected for the solution.

%\centerline{\em
%%for a general $M$, preconditioning $ M^{-1}Au=M^{-1}v$ accelerates convergence of an iterative method\\
%%\mbox{  but to a solution that has a relative error of order }
%if we can apply preconditioning transformation $M^{-1}$  accurately, i.e. satisfying (\ref{eq:xhatbound2}),}
%
%
%\centerline{\em then the preconditioning is capable of improving solution accuracy.
%}

%Thus we are interested in identifying special type of preconditioner $M$, other than the diagonal matrices, for which the preconditioned system can be formulated accurately.

Let $A=M+K$ where $K$ is small in norm and $M$ is such that there is an inverse-equivalent algorithm for inverting $M$. Then using $M$ as a preconditioner,
we form the preconditioned system % $M^{-1}A u=M^{-1}v$
\begin{equation}\label{eq:newp}
Bx = c, \;\;\mbox{ where }\;\; B:=I+M^{-1}K, \;\; c:=M^{-1}b.
\end{equation}
This system may be formed explicitly or implicitly depending on whether we solve it by a direct or an  iterative method respectively, but it is important that $B$ or its product with vectors is formed in the way as given in (\ref{eq:newp}).  We   call this process accurate preconditioning and we will show that solving the well-conditioned system  (\ref{eq:newp}) by any backward stable algorithm leads to an inverse-equivalent accurate solution (\ref{eq:invacc}). Namely,
%\begin{equation*}
%\framebox{\parbox{15cm}
{\em
accurate preconditioning with an inverse-equivalent algorithm for inverting $M$ is an inverse-equivalent algorithm for $A$.}
%}
%\end{equation*}

The following two subsections provide  detailed analysis by considering  solving (\ref{eq:newp}) first using a direct method and then using an iterative one.
%Note that the final solution accuracy of this approach should depend directly   on how well-conditioned $B$ is and indirectly on how small $K$ is.

\subsection{Direct Method for Preconditioned Systems}
We consider forming (\ref{eq:newp}) explicitly and then solving it by a backward stable   direct method such as the Gaussian elimination with partial pivoting. In this regard, we first need to compute $M^{-1}K$ column by column by solving $n$ linear systems. Assume that these linear systems are solved by an inverse-equivalent algorithm for $M$. Then, each column of the computed result of $M^{-1}K$ has inverse-equivalent accuracy. We denote the computed result as $\what Z$ and it satisfies
$\|\what Z - M^{-1}K\| \le {\cal O} (\bu)  \|M^{-1}\| \|K\|.$  Furthermore the coefficient matrix $B =I+M^{-1}K$ is computed as
$fl(I+\what Z)$, which has an error term bounded by $ \bu  (1+\|\what Z\|)$ by  (\ref{eq:xpy}). Combining the two error terms together and denoting the final computed result $fl(I+\what Z)$ as $\what B$, we can write the total error as
\begin{equation}\label{eq:Bhat}
\what B %:= fl(I+M^{-1}K)
=  I+M^{-1}K +E = B+E,
\;\;\mbox{ with } \;\; \|E\| \le {\cal O} (\bu) (1+ \|M^{-1}\| \|K\|).
\end{equation}
Similarly, the computed result of $M^{-1}b$, denoted by $\hat c:= fl(M^{-1}b) $ satisfies
\begin{equation}\label{eq:chat}
\|\hat c -c\| \le {\cal O} (\bu)  \|M^{-1}\| \|b\|.
\end{equation}

\begin{theorem}\label{thm:dense}
Let  $A=M+K$ with $A$ and $M$ being invertible and let $Ax=b$.  Assume that there is an inverse-equivalent algorithm for inverting $M$
so that the computed results of $B := I+M^{-1}K$ and $c :=M^{-1}b$, denoted by $\what B $ and $\hat c $ respectively,
%$\hat c := fl(M^{-1}b) $ and $\what B := fl(I+M^{-1}K)$,
satisfy (\ref{eq:Bhat}) and (\ref{eq:chat}).
%\begin{equation}\label{eq:Bhat}
%\what B %:= fl(I+M^{-1}K)
%=  I+M^{-1}K +E
%\;\;\mbox{ with } \;\; \|E\| \le {\cal O} (\bu) (1+ \|M^{-1}\| \|K\|).
%\end{equation}
Let $\hat x$ be the computed solution  to $\what B x=\hat c $  by a backward stable algorithm so that $\hat x$ satisfies
\begin{equation}\label{eq:backward}
 (\what B +F) \hat x=\hat c, \;\;\mbox{ with } \;\;\frac{\|F\|}{\|\what B\|} \le {\cal O} (\bu).
\end{equation}
Let $\delta:=(\|E\|+\|F\|)\|B^{-1}\|$ % \le {\cal O} (\bu) \|B^{-1}\|(1+ \|\what B\| +\|M^{-1}\| \|K\|) $
and assume that $\delta <1$. Then
\begin{equation}\label{eq:m1}
    \frac{\|\hat x - x \|}{\|A^{-1}\|\|b\|} \le {\cal O} (\bu) \frac{\kappa(B)}{1-\delta} \left( 4 +  \frac{\|K\| \| x\|}{\|b\|} \right).
\end{equation}
In particular, if $\|M^{-1}\|\|K\|  <1 $, then
\[
    \frac{\|\hat x - x \|}{\|A^{-1}\|\|b\|} \le   \frac{ {\cal O} (\bu)}{(1-\delta)(1-\|M^{-1}\| \|K\|)^2 }  .
\]
\end{theorem}
\begin{proof}
First, let  $f =\hat c -c$ or $\hat c = c+f$. Then
$\|f\| \le {\cal O} (\bu)  \|M^{-1}\| \|b\|$. Let $\widetilde B= \what B +F  = B +E +F $ and rewrite (\ref{eq:backward}) as $\wtilde B  \hat x =  c +f.$   From $(\|E\|+\|F\|)\|B^{-1}\|<1$, it follows that $\wtilde B$ is invertible and
\begin{equation}\label{eq:btilde}
\|\wtilde B^{-1}\| \le \frac{\| B^{-1}\|}{1-(\|E\|+\|F\|)\|B^{-1}\|} =
 \frac{\| B^{-1}\|}{1-\delta}
 %\| B^{-1}\|(1- {\cal O} (\bu)  \|M^{-1}\| \|K\| \|B^{-1}\| -{\cal O} (\bu)  \|B\| \|B^{-1}\|)
\end{equation}
We also have
\begin{equation}\label{eq:M}
\|M^{-1}\| = \|B A^{-1}\| \le \|B\| \|A^{-1}\|
\end{equation}
and
\begin{equation}\label{eq:2bb}
\| B^{-1}\| \|\what B \| \le \| B^{-1}\| (\|  B \|+\|E\|)  \le \| B^{-1}\| \|  B \|+ \delta \le 2 \|B^{-1}\| \|  B \|
\end{equation}
Furthermore, using
\begin{equation}\label{eq:one}
1=\|I\|=\|  B - M^{-1} K\| \le \|B\|+ \|M^{-1}\| \|K\|,
\end{equation}
we can bound  (\ref{eq:Bhat}) as
\begin{equation}\label{eq:E}
\|E\| \le   {\cal O} (\bu) ( \|B\|+ 2\|M^{-1}\| \|K\|) = {\cal O} (\bu) ( \|B\|+ \|M^{-1}\| \|K\|)
\end{equation}
where in  the last equality we have combined the coefficient $2$ of $\|M^{-1}\| \|K\|$ into ${\cal O} (\bu)$ (i.e. $2{\cal O} (\bu) =   {\cal O} (\bu)$ with our notation).
Now, clearly $Bx=c$ and then $ \wtilde B  x=  c +E x + Fx$. Combining this with  $\wtilde B \hat x=  c +f$, we have
\[
 \hat x - x =-\wtilde B^{-1}E   x -\wtilde B^{-1}F  x +\wtilde B^{-1} f.
 \]
Bounding the above and using (\ref{eq:E}), (\ref{eq:backward}), (\ref{eq:btilde}), (\ref{eq:M}), and (\ref{eq:2bb}), we have
\begin{eqnarray*}
\|\hat x - x \| &\le& \|\wtilde B^{-1}\| \|E \| \| x\| +\|\wtilde B^{-1}\| \|F\| \| x\| +\|\wtilde B^{-1}\| \| f\| \\
    &\le&  {\cal O} (\bu)\|\wtilde B^{-1}\| (\|B\|+\|M^{-1}\| \|K\|) \| x\| \\
    & & + {\cal O} (\bu)\|\wtilde B^{-1}\| \|\what B \| \| x\| +{\cal O} (\bu)\|\wtilde B^{-1}\| \|M^{-1}\| \|b\| \\
    &\le&  \frac{{\cal O} (\bu) }{1-\delta} \| B^{-1}\| \|B\| \| x\|+ \frac{{\cal O} (\bu) }{1-\delta} \| B^{-1}\| \|M^{-1}\| \|K\|  \| x\|  \\
    & & + \frac{{\cal O} (\bu) }{1-\delta}  \| B^{-1}\|\|\what B \| \| x\|+\frac{{\cal O} (\bu) }{1-\delta}  \| B^{-1}\|\|M^{-1}\| \|b\| \\
    &\le&  \frac{{\cal O} (\bu) }{1-\delta} ( \| B^{-1}\| \|B\| \| x\| + \| B^{-1}\|\|B\| \|A^{-1}\| \|K\| \| x\| \\
    & & + 2\| B^{-1}\| \|  B \| \| x\| +\| B^{-1}\| \|B\| \|A^{-1}\| \|b\|  ) \\
    &\le& \frac{{\cal O} (\bu) }{1-\delta} \| B^{-1}\| \|B \| \left(3\|x\| +  \|A^{-1}\| \|b\|\frac{\|K\| \| x\|}{\|b\|} + \|A^{-1}\| \|b\|\right) \\
     &\le&  \frac{{\cal O} (\bu)}{1-\delta} \kappa(B)\left( 4 +  \frac{\|K\| \| x\|}{\|b\|} \right) \|A^{-1}\|\|b\|.
\end{eqnarray*}
where we have used $\|x\| \le \|A^{-1}\|\|b\|$ in the last inequality. This proves (\ref{eq:m1}).

Finally, if  $\|M^{-1}\| \|K\|  <1$, then $B=I+M^{-1}K$ satisfies $\|B\|\le 1+ \|M^{-1}\| \|K\| $ and
$\|B^{-1}\| \le \frac{1}{1-\|M^{-1}\| \|K\|}$. Thus, it follows from $A^{-1}=B^{-1}M^{-1} $ that
\[
\frac{\|K\| \| x\|}{\|b\|}\le \|K\| \|A^{-1}\|\le \|K\| \|M^{-1}\|   \|B^{-1}\| \le \frac{\|K\| \|M^{-1}\| }{1-\|K\| \|M^{-1}\| }
\]
Thus
\[
{\kappa(B)} \left( 4 +  \frac{\|K\| \| x\|}{\|b\|} \right)\le \frac{1+ \|M^{-1}\| \|K\|}{1-\|M^{-1}\| \|K\|}\frac{4 - 3\|M^{-1}\| \|K\|}{ 1-\|M^{-1}\| \|K\| }
\le \frac{5}{ (1-\|M^{-1}\| \|K\|)^2 }
\]
where we have used $(1+ \|M^{-1}\| \|K\|)(4 - 3\|M^{-1}\| \|K\|)\le 4+\|M^{-1}\| \|K\| \le 5$.
Substituting this into (\ref{eq:m1}) and combing the factor 5 into the  ${\cal O} (\bu)$ term, we obtain the second bound of the theorem.
\end{proof}

The second bound of the theorem shows that we can obtain an inverse-equivalent solution if $\|M^{-1}\| \|K\|$ is bounded away from 1. Note that $\delta=(\|E\|+\|F\|)\|B^{-1}\|  \le {\cal O} (\bu) \|B^{-1}\|(1+ \|\what B\| +\|M^{-1}\| \|K\|) $ can be expected to be much smaller than 1 and hence the factor $(1-\delta)^{-1}$ is insignificant. When $\|M^{-1}\| \|K\|\ge 1$, only the first bound (\ref{eq:m1}) holds, which implies that the inverse-equivalent accuracy of the solution may deteriorate by a factor of $\kappa(B)$ or $ \frac{\|K\| \| x\|}{\|b\|} $. Such a dependence on $\kappa(B)$ and $K$ is expected however, as otherwise there would be   inverse-equivalent algorithm  for any $A$.

\subsection{Iterative Method for Preconditioned Systems}
For large scale problems, we are more interested in solving the preconditioned system $Bx=c$ by an iterative method. In general, the accuracy of the approximate solution obtained by an iterative method for $Ax=b$ is obviously limited by the accuracy of the matrix-vector multiplication $Av$. Namely, the residual $\|b-A \what x\|$ of an approximate solution $\what x$ computed in a floating point arithmetic is at best of order $\bu N \|A\| \|\what x\|$. A careful implementation,   possibly using residual replacements \cite{vdvy}, can ensure that the residual converges with this level of accuracy. Note that such a solution $\what x$ is backward stable (see \cite[Theorem 2.2]{demmel97}). We first briefly discuss some related results on the best accuracy that can be achieved.

Most iterative methods update approximate solutions and the corresponding residuals at each iteration by general formulas of the forms $x_k = x_{k-1}+q_k$ and $r_k = r_{k-1}-Aq_k$. In a convergent iteration, the best residual $\|b-A x_k\|$  one may obtain in finite precision is determined by the deviation between the computed (or updated) residual $r_k$, which is the one computed in an algorithm through the updating formula $r_k = r_{k-1}-Aq_k$, and the true residual defined as $b-A x_k$ for  $x_k$ that is computed through $x_k = x_{k-1}+q_k$.  This deviation phenomenon of the two kinds of residuals has been extensively studied; see \cite{svdvf,greenbaum,gutknecht,svdv,vdvy} and the
references cited therein.   Typically, the computed residuals $r_k$ of a convergent method maintains the theoretical convergence property (e.g. monotonicity)  even in a floating point arithmetic and can decrease arbitrarily close to 0, but the true residuals $b-A x_k$ will stagnate  at some level. This  deviation of the two residuals is due to the roundoff errors at each step, the most significant of which,  among others, is ${\cal O}(\bu) N\|A\|  \|q_k \|$ incurred in computing $Aq_k$, where $N$ is a constant associated with the error in $fl(Av)$ as defined in (\ref{eq:dot2}). Then, for $x_L$ at step $L$, the deviation is made up of the accumulated deviations over $L$ iterations, i.e. ${\cal O}(\bu) \sum_{k=1}^K N\|A\|  \|q_k \|$ which, since $x_L =fl( \sum_{k=1}^L q_k) $, is at least ${\cal O}(\bu) N\|A\|\|x_L \|$, the error incurred in computing $fl(A x_L)$.

Indeed, the accumulated roundoff errors ${\cal O}(\bu) \sum_{k=1}^L N \|A\|  \|q_k \|$, and hence the true residual, may be much larger than ${\cal O}(\bu) N \|A\|  \|x_L \|$ if there are large intermediate iterates $q_k$, which occur often in nonsymmetric solvers such as BiCG and CGS.
In that case, a residual replacement strategy \cite[Algorithm 3]{vdvy} has been developed that replaces the computed residual by the true residual at some selected steps so that its convergence property remains intact but the deviation of the two residuals is reset to 0.
  %that can maintain  the deviation of the two kinds of residuals to the level of ${\cal O}(\bu) \|A\|\|x_L \|$.
%the errors in the matrix-vector multiplication near the end of iteration only.
Indeed, it is shown in \cite[Theorem 3.6]{vdvy} that if an iterative method for solving $Ax=b$ is implemented with the residual replacement and
the algorithm terminates at step $L$ with the computed residual satisfying
$\|{r}_{L} \| < \bu \|A\| \|x_L \|$, then  the true residual
$\| b-A x_L \|$
%   \le    \|r_K\| + \bu  N \| A \| \|{x}_{K} \|  /
%   (1- \bu L )
%   + O({\bf u}^2) \\
% & \sim  &
will be in the order of  %agree with $r_K$ to the level of
$\bu  N \| A \| \|{x}_{L}\|$.
%We note that this result (\cite[Theorem 3.6]{vdvy}) is not dependent on how $Av$ is computed. Namely, if the matrix-vector multiplication $u=Av$ %involves some  more complicated operations different from simple matrix-vector multiplication, e.g.
%is computed through some function subroutine, as long as its computed result $\hat u$ satisfies the same bound as  (\ref{eq:dot}), i.e.
%$
%\|\hat u - u\| \le  N \| A \| \|u \|
%$
%for some $N$, then the true residual
%$\| b-A x_L \|$  can  converge to  $\bu  N \| A \| \|{x}_{K} \|$.
%%We state the key result there but omit the algorithmic details.
%%\begin{theorem} (\cite[Theorem 3.6]{vdvy}) Consider an iterative method for solving $Ax=b$ implemented with the residual replacement strategy of Algorithm 3 in \cite{vdvy} with  $\epsilon$ being the threshold for replacement criterion.
%%Assume that the algorithm terminates at step $n=K$ with the computed residual satisfying
%%$\|{r}_{K} \| < \bu \|A\| \; \|x_L \|$.
%%Let $m$ be the number of the last residual replacement iteration step before
%%termination andthe .
%%If
%%\begin{equation}\label{eq:18}
%%L = (K-m+1)  (1+2N) \|A\|  \| A^{-1} \| (1 + 3/ \epsilon) < 1/\bu,
%%\end{equation}
%%then
%%\begin{eqnarray*}\label{eq:19}
%%\| b-A x_L \|  & \le &  \|r_K\| + \bu  N \| A \| \|{x}_{K} \|  /
%%   (1- \bu L )
%%   + O({\bf u}^2) \\
%% & \sim  & \bu  N \| A \| \|{x}_{K} \|  .
%%\end{eqnarray*}
%%\end{theorem}

Now, consider solving the preconditioned system (\ref{eq:newp}) by such an iterative method. To determine the accuracy that can be obtained from solving this well-conditioned system, we first analyze the accuracy of computing matrix-vector multiplication $B v$.

\begin{lemma}\label{lm:Bv}
Let $B$ be defined in (\ref{eq:newp}) and consider computing $Bv = v+ M^{-1}K v$ as in this expression for any $v\in \R^n$.  Assume that there is an inverse-equivalent algorithm for inverting $M$. If $M^{-1}K v$  is computed by the inverse-equivalent
algorithm and if we denote the final computed  result  of $Bv$ by $fl(Bv) $, then
\begin{equation}\label{eq:Bv}
\|fl(Bv) - Bv \| \le  {\cal O} (\bu) (1+\|M^{-1}\| \|K\|) \|v\|.
\end{equation}
\end{lemma}
\begin{proof}
Let $u:=Bv$ and denote the final computed  result $fl(Bv) $ by  $\hat u$.
To compute $Bv$,  we first compute $Kv$ to get $fl(Kv)= Kv+e_1$ with $|e_1| \le n \bu |K| |v|$. Then computing $M^{-1} fl(Kv)$ by the inverse-equivalent algorithm, the computed result, denoted by $\hat w$, satisfies
\begin{eqnarray*}
\|\hat w- M^{-1} fl(Kv)\| &\le& {\cal O} (\bu) \|M^{-1}\| \|fl(Kv)\| \\
  &\le& {\cal O} (\bu) \|M^{-1}\| (\|K\| \|v\|+ {\cal O} (\bu) \|K\| \|v\|) \\
   &\le& {\cal O} (\bu) \|M^{-1}\| \|K\| \|v\|.
\end{eqnarray*}
%where $1+n \bu \le 1+n$ is absorbed into the ${\cal O} (\bu)$ term.
Let $e_2 = \hat w- M^{-1} Kv$. Then
\begin{eqnarray*}
\|e_2\| &=& \|\hat w- M^{-1} fl(Kv) +M^{-1} e_1\| \\
 &\le& {\cal O} (\bu) \|M^{-1}\| \|K\| \|v\| +{\cal O} (\bu)  \|M^{-1}\| \|K\| \|v\| \\
&=& {\cal O} (\bu) \|M^{-1}\| \|K\| \|v\|.
\end{eqnarray*}
Now, $\hat u = fl(v+\hat w) = v+\hat w +e_3$ with $|e_3| \le   \bu (|v|+|\hat w|).$ Then
\begin{eqnarray*}
\|e_3\| &\le&  {\cal O} (\bu)  \|v\|+ {\cal O} (\bu) (\|M^{-1} Kv\|+\|e_2\|)  \\
    &\le& {\cal O} (\bu) ( \|v\|+    \|M^{-1}\| \|K\| \|v\|+  {\cal O} (\bu)  \|M^{-1}\| \|K\| \|v\|)  \\
    &=&{\cal O} (\bu) (\|v\|+  \|M^{-1}\| \|K\| \|v\|).
\end{eqnarray*}
Finally, we have
$\hat u = v+ M^{-1} Kv +e_2 +e_3 = u +e_2 + e_3$ and
\[
\|\hat u - u \| \le \|e_2\| +\| e_3\| \le  {\cal O} (\bu) (1+\|M^{-1}\| \|K\|) \|v\|.
\]
\end{proof}

Now, when applying some convergent iterative method to the system (\ref{eq:newp}), using the residual replacement strategy if necessary, the true residual $\| c-B x_L \|$ is expected to converge to $ {\cal O} (\bu) (1+\|M^{-1}\| \|K\|) \|v\|$.  The next theorem demonstrate that such a solution has an inverse-equivalent accuracy. Note that since (\ref{eq:newp}) is  well-conditioned, most iterative methods should have fast convergence. In that case, the error accumulations are insignificant and the residual replacement is usually not necessary in practice.

\begin{theorem}\label{thm:iter}
Consider solving  (\ref{eq:newp}) by an iterative method where the matrix-vector product $Bv = v+ M^{-1}K v$ is computed by an inverse-equivalent algorithm for inverting $M$.  Assume that the iterative method produces an approximate solution $x_L$ with $\| c-B x_L \| \le  {\cal O} (\bu) (1+\|M^{-1}\| \|K\|) \|v\|$ and $\|b-Ax_L\| \le  \|b\|$. Then
\begin{eqnarray*}
\frac{\| x- x_L \|}{ \|A^{-1}\| \|b\|  }  &\le& {\cal O} (\bu)\kappa (B) \left( 1 +  \frac{\|K\| \|x_L\|}{\|b\|} \right) \\
   &\le&  {\cal O} (\bu)  \kappa (B) \left( 1 +  2 \|A^{-1}\| \|K\|  \right) .
\end{eqnarray*}
\end{theorem}
\begin{proof}
First we note that $x_L=x-A^{-1}(b-Ax_L)$ and then
\[
\|x_L\| \le \|x\| + \|A^{-1}\| \|b-Ax_L\| \le \|x\|+\|A^{-1}\| \|b\| \le 2 \|A^{-1}\| \|b\|.
\]
As in the proof of Theorem \ref{thm:dense}, we have (\ref{eq:one}). Then
\[
\| c-B x_L \| \le {\cal O} (\bu) (\|B\| \|x_L\| +  2 \|M^{-1}\| \|K\| \|x_L\|).
\]
We now bound $x-x_L = B^{-1} (c-B x_L)$ as
\begin{eqnarray*}
\| x- x_L \|  &\le&  {\cal O} (\bu) \| B^{-1}\|(   \|B\| \|x_L\|+  2\|M^{-1}\| \|K\| \|x_L\|) \\
    &\le& {\cal O} (\bu) \| B^{-1}\|(   2 \|B\|   \|A^{-1}\| \|b\|+ 2\|B\| \|A^{-1}\| \|K\| \|x_L\|) \\
    &=& {\cal O} (\bu) \|A^{-1}\| \|b\| \kappa (B) \left(   1 +  \frac{\|K\| \|x_L\|}{\|b\|} \right),
\end{eqnarray*}
where we have used (\ref{eq:M}) and combine the factor 2 into the ${\cal O} (\bu)$ term. This proves the first bound.  Bounding $\|x_L\|$ by $2 \|A^{-1}\| \|b\|$ again, % and combining the factor 2 into $ {\cal O} (\bu)$ term,
we obtain the second bound of the theorem.
\end{proof}

The theorem shows that the inverse-equivalent accuracy is also achieved when using an iterative method for the preconditioned system.

\subsection{Accurate Inversion of Preconditioner}
The key requirement of the accurate preconditioning method is that there is an inverse-equivalent algorithm for inverting the preconditioner $M$. This is obviously the case if the inverse $M^{-1}$  is explicitly available. More generally, if a preconditioner $M$ has  an {\em accurate}  rank-revealing decomposition (RRD), then the solution to $Mx=b$ computed from the RRD is inverse-equivalent. The {\em accurate} rank-revealing decomposition is introduced by Demmel et. al. \cite{demmel99b} to accurately compute the singular value decomposition of a matrix.  Here is its definition.

\begin{definition}\label{arrf}
{\rm (See \cite{demmel99b})}
A factorization $A=XDY$ of $A\in \R^{m\times n}$ with $m\ge n$ is said to be rank-revealing if $X\in \R^{m\times r}$ and $Y\in \R^{r\times n}$ are  well-conditioned and $D\in \R^{r\times r}$ is diagonal and invertible, where $r \le \min\{m, n\}$. Consider an algorithm for computing a rank-revealing decomposition $A=XDY$ and let $\widehat X$, $\widehat D$, and $\widehat Y$ be the computed factors. We say $\widehat X \widehat D \widehat Y$ is an {\em accurate rank-revealing decomposition} of $A$  if $\widehat X$ and $\widehat Y$ are normwise accurate and $\widehat D$ is entrywise accurate, i.e.,
\begin{equation}\label{eq:arrf}
\frac{\|\widehat X -X\|}{\|X\|} \le  \bu p(n); \;\;
\frac{\|\widehat Y -Y\|}{\|Y\|} \le  \bu p(n); \;\;
\;\; \mbox{ and } \;
|\widehat D-D|  \le  \bu p(n)|D|,
%\frac{\|\widehat X -X\|}{\|X\|} \le  {\cal O} (\bu), \;\;% \bu p(n); \;\;
%\frac{\|\widehat Y -Y\|}{\|Y\|} \le   {\cal O} (\bu), %\bu p(n); \;\;
%\;\; \mbox{ and } \;
%|\widehat D-D|  \le  {\cal O} (\bu). % \bu p(n)|D|,
\end{equation}
where $p(n)$ is a polynomial in $n$.
\end{definition}

As noted in \cite{demmel99b}, the precise meaning of ``well-conditioned" in the definition is not important as all related results involving this will be stated in terms of the condition numbers $\kappa(X)$ and $\kappa(Y)$, but in general, it refers to matrices with a  condition number within a modest bound dependent on the problem at hand.

For our purpose, we consider $n\times n$ invertible  matrices, i.e. $r=n$. Then, if $A$ has an accurate RRD, it is shown by Dopico and Molera \cite{domo12} that using it to solve linear systems gives an inverse-equivalent algorithm.   We state this result in the following theorem.

\begin{theorem}\label{thm:dopico}
{\rm (\cite[Theorem 4.2]{domo12})}
Let $\widehat X$, $\widehat D$, and $\widehat Y$ be the computed factors of a rank-revealing decomposition of $ A=XDY$  and assume that they satisfy (\ref{eq:arrf}).
%\[
%\frac{\|\widehat X -X\|}{\|X\|} \le  \bu p(n); \;\;
%\frac{\|\widehat Y -Y\|}{\|Y\|} \le  \bu p(n); \;\;
%\;\; \mbox{ and } \;
%|\widehat D-D|  \le  \bu p(n)|D|,
%\]
%where  $p(n)$ is a polynomial of $n$ and $ X, D, $ and $Y$ are the corresponding exact factors.
Assume also that the systems $X s = b$ and $Y x = w$ are solved with a backward
stable algorithm that when applied to any linear system $Bz = c$,  computes a
solution $\hat z $ that satisfies $(B+\Delta B) \hat z = c$; with $\|\Delta B\| \le \bu q(n) \|B\|$
where $q(n)$ is a polynomial in $n$ such that $q(n) \ge 4 \sqrt{2} /(1-12\bu). $ Let
$g(n) := p(n)+q(n)+\bu p(n)q(n).$ Then, if $\hat x$ is the computed solution of $Ax = b$ through solving
\[
\widehat X y =b; \; \; \widehat D z= y; \; \mbox{ and } \; \widehat Y x = z,
\]
and if $\bu g(n)\kappa (Y)<1$ and $\bu g(n) (2+\bu g(n))\kappa(X) <1,$ then
\begin{eqnarray*}
 {\|\widehat x -x\|}  &\le&  \frac{\bu g(n) }{1-\bu g(n)\kappa (Y) } \left( \kappa (Y)+
   \frac{1+(2+\bu g(n)) \kappa(X) }{1-\bu g(n)(2+\bu g(n)) \kappa(X) }{\|A^{-1}\| \|b\|}  \right)\\
    &=&  \left( \bu g(n) + {\cal O}(\bu^2)\right) \max\{\kappa(X), \kappa(Y)\}  {\|A^{-1}\| \|b\|}.
\end{eqnarray*}
\end{theorem}

%For the convenience of later uses, we rewrite (\ref{eq:xhatbound}) in the 2-norm as
%\begin{equation}\label{eq:xhatbound2}
%\|\widehat x -x\|\le {\cal O} ( \bu) \|A^{-1}\| \|b\|.
%\end{equation}

%Therefore, $\what x$ is inverse-equivalent as long as $\kappa(X), \kappa(Y)$ are modest constants.
%To use this result, we construct preconditioners for which an accurate rank-revealing decomposition can be computed.

Several classes of matrices have been shown to have accurate RRD by Demmel et. al. \cite{demmel99b}, which include
graded matrices, total signed compound matrices such as acyclic matrices, Cauchy matrices, totally positive  matrices, diagonally scaled totally unimodular matrices, and matrices arising in certain simple finite element problems.
Diagonally dominant matrices  have also been shown to have accurate rank-revealing decomposition; see \cite{axy2,ddy14b,ye08}.  Specifically, in \cite[Algorithm 1]{ye08},  a variation of the Gaussian elimination is developed
to compute  an accurate $LDU$ factorization %from  $A$ and its diagonally dominant parts
that is shown to be an accurate rank-revealing decomposition. The computational cost of this accurate $LDU$ algorithm is about the same as the standard Gaussian elimination.
Since discretizations of differential equations are often close to being diagonally dominant, we can construct a diagonally dominant preconditioner, for which the accurate  $LDU$ factorization provides an inverse-accurate algorithm. This will be used in our numerical examples in \S \ref{sec:example}.

%Thus, matrices from any of these classes and those that can be well preconditioned by them can be solved  with an inverse-accurate solution.
%
We remark that if two matrices $A_1$ and $A_2$ both have accurate rank-revealing decomposition, then solving $A_1 A_2 x=b$ through $A_1 y=b$ and $A_2 x= y$ will produce an inverse-equivalent solution provided $\|A_1^{-1}\| \|A_2^{-1}\|/\|(A_1 A_2)^{-1}\|$ is a modest number; see \cite{ye17}. In particular, we may also consider a preconditioner that is a product of diagonally dominant matrices; see Examples 2 and 4 in \S \ref{sec:example}.

\section{Application to  Eigenvalue Problems}
\label{sec:eigen}
In this section we discuss an application of inverse-equivalent algorithms to computing a few smallest eigenvalues (in absolute value) of a matrix through accurate inverses.
%Let $\lambda_1, \lambda_2, \cdots,  \lambda_n$ be the eigenvalues of an $n\times n$ invertible matrix $A$  with $|\lambda_1| \le |\lambda_2| \le \cdots \le |\lambda_n|$.

In general, the relative accuracy of the computed smallest eigenvalue of a matrix in finite precision depends on the condition number $\kappa_2 (A)$. To illustrate, we consider an $n\times n$  symmetric positive definite matrix $A$. Let $\lambda_1 \le \lambda_2 \le \cdots \le \lambda_n$ be its eigenvalues.
A  backward stable algorithm computes an
approximate eigenvalue-eigenvector pair $(\widehat \lambda_i, \widehat x_i)$ with $\|\what x_i\|=1$ such that the residual  $ \|A \widehat x_i - \widehat \lambda_i \widehat x_i\|  $ is of order $\bu \|A\| $.
%%%${\|A \widehat x_i - \widehat \lambda_i \widehat x_i\| =
%%%\over \|A\| \|\widehat x_i\|} = {\cal O} (\bu)$.
%${\cal O} (\bu) \|A\| $ (see \cite{ye15} for example). % (see \cite[p.293]{parlett} for example).
%%For both the dense and iterative eigenvalue algorithms then, the error on the computed eigenvalue $\widehat\lambda_i$ is at best
Then,  $|\widehat\lambda_i - \lambda_i | \le {\cal O} (\bu) \|A\|$ and hence
%This also holds if $\widehat\lambda_i$ is obtained by an iterative algorithm; see \S 2.
\begin{equation}\label{eq:relerror}
\frac{|\widehat\lambda_i - \lambda_i |}{\lambda_i }
\le {\cal O} (\bu) \frac{\lambda_n }{ \lambda_i}
%\approx
%  \begin{cases}
%{\cal O} (\bu)  & \text{if}\;\; \lambda_i\approx \lambda_n\\
%{\cal O} (\bu) \kappa (A) & \text{if}\;\; \lambda_i\approx \lambda_1
%  \end{cases}.
\end{equation}
It follows that {\em larger eigenvalues}
(i.e.  $\lambda_i\approx \lambda_n$) are computed to the accuracy of machine precision, but
for {\em smaller eigenvalue} (i.e.  $\lambda_i\approx \lambda_1$),  a relative error of order ${\cal O} (\bu) \kappa (A)$ is expected.

Since the larger eigenvalues can be computed accurately, to compute a few smallest eigenvalues of an ill-conditioned matrix, we may compute correspondingly a few largest eigenvalues of $A^{-1}$. However, a difficulty  with this approach is that,  $A^{-1}$, or its multiplications on vectors, can not be computed accurately since $A$ is assumed to be ill-conditioned. For diagonally dominant matrices, this can be remedied by  using the accurate $LDU$ factorizations  \cite{ye08,ye09,ye17}.  Specifically, for large scale problems,  we  apply in  \cite{ye17} the Lanczos method to  $A^{-1}$ or simply use the inverse iteration and compute its largest eigenvalue $\mu_1=\lambda_1^{-1}$. At each iteration, the accurate $LDU$ factorizations is used to compute the matrix-vector product   $A^{-1}v$ (i.e. solving $Au=v$), which produces a solution  that would be equivalent to the one produced by multiplying the exact $A^{-1}$ with $v$. Hence the resulting residual error will be of order   $\bu \|A^{-1}\|_2 =\bu \mu_1,$ which implies  a relative error  for $\mu_1$ in the order of machine precision.   Finally $\lambda_1 = \mu_1^{-1}$ is computed accurately.

Now, consider a general symmetric matrix $A$ that can be preconditioned by  a diagonally dominant matrix. Then using the accurate preconditioning scheme of \S \ref{sec:inv}, we can form  $A^{-1}v$ (i.e. solving $Au=v$) with the inverse-equivalent accuracy. Then in the same way as discussed above, a few largest eigenvalues in absolute values can be computed accurately for $A^{-1}$, from which a few smallest eigenvalues in absolute values for $A$ are computed accurately.

The same discussion can also be extended to  nonsymmetric matrices  with the modification of the bound (\ref{eq:relerror}) as
\[
\frac{|\widehat\lambda_i - \lambda_i |}{|\lambda_i| }
\approx {\cal O} (\bu)\frac{1 }{ c_i} \frac{\|A\| }{ |\lambda_i|}
\]
where $c_i$ is the cosine of the angle between the left and the right eigenvectors of $A$ corresponding to $\lambda_i$; see \cite[Theorem 4.4]{demmel97}. The additional factor $1/c_i$ defines the sensitivity caused by nonnormality of the matrix, which is also studied through pseudospectra (see \cite{tref}). This factor is not changed with the inverse. Thus, when  $A^{-1}v$ is computed with inverse-equivalent accuracy, a few largest eigenvalues of  $A^{-1}$, and then the corresponding eigenvalues of $A$, are computed with an accuracy  independent of the condition number $\kappa (A)$. However, the accuracy is still expected to depend on $1/c_i$. See \cite[Theorem 6.3]{ddy14a} for some some related discussions for diagonally dominant matrices.

%This is used to accurately compute a few smallest eigenvalues of certain differential operators such as 1-dimensional biharmonic operators.
%It is important to note that $v$ as the right-hand side of the linear system is near the eigenvector corresponding to the smallest eigenvalue at the convergence stage. Then by (\ref{eq:bcond}), a backward stable solution should be much worse than an inverse-equivalent solution.

Finally, we discuss an application to discretizations of differential operators, which is a large source of ill-conditioned problems. For the discretization of differential eigenvalue problems, it is usually a few smallest eigenvalues that are of interest but their computed accuracy is reduced by the condition number of the discretization matrix. For operators involving high order differentiations such as biharmonic  operators, the matrix may easily become extremely ill-conditioned and then little accuracy may be expected of the computed eigenvalues; see \cite{bjtj97,bjtj99,ye17} and \S \ref{sec:example}.

In \cite{ye17}, we have used the accurate inverse approach to accurately compute a few smallest eigenvalues of a differential operator whose discretization matrix is diagonally dominant. For differential operators whose discretizations are not diagonally dominant, they can often be preconditioned by a diagonally dominant matrix.
%Our method can be used to compute a few smallest eigenvalues of differential operators. This is important as discretizations of differential operators may not diagonally dominant, but might be close to  and hence be preconditioned by a diagonally dominant matrix. This would significantly broaden the class of operators that we can accurately compute the eigenvalues.
For example, consider  the finite difference discretization of the convection-diffusion operator
\[
- \bigtriangleup u + \beta  u_x + \gamma  u_y    = \lambda u \quad
\mbox{ on } \quad (0,1)^2;
\]
with the homogeneous Dirichlet boundary condition. The discretization matrix $A$
is not diagonally dominant, but the discretization of the diffusion operator $- \bigtriangleup $ is. The convection operator is dominated by the diffusion operator, if $\beta, \gamma$ are not too large. Then, the discretization matrix can be well preconditioned by that of the diffusion operator. Hence, using accurate preconditioning,  we can accurately compute a few smallest eigenvalues of the convection-diffusion operator. See Examples 3  in \S \ref{sec:example} for some numerical results.

The convection-diffusion operator is just one example of differential operators whose discretization may be preconditioned by a diagonally dominant matrix. It will be interesting to study other differential operators with such properties but we leave it to  a future work.

%Now, suppose a matrix $A$ is not diagonally dominant but can be written as a product of two or more diagonally dominant matrices, say $A=A_1 A_2$. Then we can factorize $A_1$ and $A_2$ using Algorithm \ref{alg:lu} and then solve $Au=v$ through $A_1 w= v$ and then $A_2 u  =w$ as in (\ref{eq:xhat}). Let $\widehat w$  and $\widehat u$ be the computed solutions to $A_1 w= v$ and   $A_2 u  =\widehat w$ respectively. Using (\ref{eq:xhatbound2}), they satisfy
%\[
%\|\widehat w -w\| \le {\cal O} (\bu) \|A_1^{-1}\| \|v\|
%\]
%and
%\begin{eqnarray*}
%\|\widehat u -u\| &\le &{\cal O} (\bu) \|A_2^{-1}\| \|\widehat w\| \\
%&\le& {\cal O} (\bu) \|A_2^{-1}\| \| w\| + {\cal O} (\bu) \|A_2^{-1}\| \|\widehat w-w\| \\
%&\le & {\cal O} (\bu) \|A_2^{-1}\| \|A_1^{-1}\| \|v\|.
%\end{eqnarray*}
%Thus, writing $\gamma =  \|A_1^{-1}\| \|A_2^{-1}\|/ \|A^{-1}\| $, we have
%\begin{equation}\label{eq:uw}
%\|\widehat u -u\|\le {\cal O} (\bu) \gamma \|A^{-1}\| \|v\|.
%\end{equation}
%So, as discussed above, the largest eigenvalue of $A^{-1}$ can be computed to an accuracy of order ${\cal O} (\bu) \gamma$, which should be satisfactory as long as $\gamma$ is not an extremely large constant.

\section{Numerical Examples}
\label{sec:example}

In this section, we present four numerical examples to demonstrate performance of the accurate preconditioning scheme. All tests were carried out on a PC in MATLAB (R2016b) with a machine precision $\bu \approx$ \texttt{2e-16}.
The first two examples concern solving linear systems and the last two examples use two similar matrices with known exact eigenvalues for the eigenvalue problems.

We consider linear systems arising in finite difference discretizations of some differential equations scaled so that the resulting matrix has integer entries. We  construct an integer solution $x$ so that $b=Ax$ can be computed exactly. Then $x$ is the exact solution. In our testing, we are interested in  systems with a random $b$, as this resembles practical situations where $b$ is usually the input data. By (\ref{eq:bcond}), a random $b$ is also likely to yield a system where an inverse-equivalent accurate solution is significantly more accurate than a backward stable solution. To construct a random integer vector $b$ with integer solution $x$, we   first construct a random vector $ {\tt b_0 =rand(n,1)}$ and set $x_0 = {\tt A\backslash b_0}$, from which we construct a scaled integer solution $x = {\tt round(x_0*1e8/norm(x_0, inf))}$, all in MATLAB functions. Then $b=Ax$ is computed exactly and is approximately a scaled random vector $b_0$.

We solve the systems by a preconditioned iterative method with the preconditioner solved by the usual Cholesky factorization and by the  accurate $LDU$ factorization (\cite[Algorithm 1]{ye08}). In all examples here, the preconditioners are symmetric; so  we actually compute the $LDL^T$ factorization which has about half of the cost. We compare the computed solutions  $\what x$ with respect to the errors
\[
\eta_{ie} := \frac{\|\what x- x\|_2}{\|A^{-1}\|_2 \|b\|_2}\;\;\mbox{ and }\;\;
\eta_{rel} := \frac{\|\what x- x\|_2}{\|x\|_2}.
\]
$\eta_{ie}$ measures the inverse-equivalent accuracy and $\eta_{rel}$ is the relative accuracy. They differ by a fixed ratio $\frac{\|x\|_2}{\|A^{-1}\|_2 \|b\|_2} \le 1$. For a backward stable solution $\what x$, $\eta_{rel}$ is approximately $\kappa_2 (A) \bu$ but for an
inverse-equivalent accurate solution, $\eta_{ie}$ is in the order of machine precision $\bu$.

\emph{Example 1.}
Consider the 1-dimensional convection-diffusion equation
\[
-  u''(x) -   u'(x)    = f(x) \quad
\mbox{ on } \quad (0, \gamma);
\]
with the Dirichlet boundary condition $u(0)=u(\gamma)=0$. Discretizing on a uniform grid of size $h=\gamma/(n+1)$ by the center difference scheme, we obtain
$A_n  =\frac{1}{h^2} T_n - \frac{1}{2h} K_n$,
%$A_n = \left( \frac{n+1}{\gamma}\right)^2 T_n - \frac{1}{2\gamma}(n+1)   K_n$,
where $T_n$ is the $n\times n$ tridiagonal matrix with diagonals being $2$ and off-diagonals being $-1$, and  $K_n$ is the skew-symmetric $n\times n$ tridiagonal matrix with $1$ on the superdiagonal above the main diagonal. To construct   $b$ and the exact solution $x=A_n^{-1}b$, we scale $A_n$ by $2\gamma^2/(n+1)$ and  use an integer value for $\gamma$ so that the resulting matrix
$2(n+1) T_n - \gamma   K_n$ has integer entries.  We then construct a random integer vector $b$ and the corresponding exact solution $x$ as discussed  at the beginning of this section.

$T_n$ is  diagonally dominant and has an accurate $LDL^T$ factorization. $A_n$ is neither symmetric nor diagonally dominant but, if $\gamma$ is not too large,  preconditioning by $ 2(n+1) T_n$   yields a well-conditioned matrix $B=I-\frac{h}{2} T_n^{-1} K_n$. We  solve the preconditioned system by  the GMRES method with the preconditioning equations solved in two ways: 1. using the Cholesky factorization of $T_n$, and 2. the accurate $LDU$ factorization of $T_n$. The GMRES is implemented with restart after 50 iterations and the stopping tolerance for relative residual is set as $\sqrt{n} \bu$. As a reference, we also solve the original system using MATLAB's division operator ${\tt A_n \backslash b}$. We compare the computed solutions $\what x$ by the three methods with respect to $\eta_{ie} $ and  $\eta_{rel}$.

In Table \ref{tbl:ex1a}, we present the results for a mildly ill-conditioned case with $n= 2^{13} -1 =8,191$ and a more ill-conditioned case with $n= 2^{19} -1 =524,287$. For each case of $n$, we test $\gamma=10^1,   10^2, \cdots, 10^6$, resulting in an $A_n$ that is increasingly not symmetric and  not diagonally dominant. In the table, in addition to the errors $\eta_{ie}$ and $\eta_{rel}$, we also present the condition numbers $\kappa_2 (A_n)$ and,  for the smaller $n$ case, $\kappa_2 (B)$ as well.
% (computed using {\tt svds}).
In the columns for accurate $LDU$ preconditioning, we also list $\rho := \|\gamma K_n\|_1 \|x\|_1 /\|b\|_1 $ which is a factor in the error bound for $\eta_{ie}$ by accurate preconditioning; see Theorem \ref{thm:iter}.
%The approximate solutions by all three methods have the relative residuals $\frac{\|b-A_n\what x\|_2}{\|A_n\|_2 \|\what x\|_2}$  approximately  $10^{-17}$; namely all three solutions are backward stable.

\begin{table}[hbt] \small
\begin{center}
\caption{\small {Example 1}: Accuracy for the three methods (${\tt A\backslash b}$, Cholesky Preconditioning, Accurate $LDU$ Preconditioning): $\eta_{ie} = \|\what x- x\|_2/(\|A^{-1}\|_2 \|b\|_2)$ and $\eta_{rel} = \|\what x- x\|_2/\|x\|_2$, and $\rho =\|\gamma K_n\|_1 \|x\|_1 /\|b\|_1$.
}
\begin{tabular}{cc|cc|ccc|ccc}
 \multicolumn{2}{c|}{} & \multicolumn{2}{c|}{ ${\tt A\backslash b}$} &  \multicolumn{3}{c|}{ Cholesky Precond.} &  \multicolumn{3}{c}{Accurate Precond.} \\
\hline
$\gamma$ & $\kappa_2 (A_n)$ &  $\eta_{ie} $ &  $\eta_{rel}$ &
$\eta_{ie}$ &  $\eta_{rel}$ & $\kappa_2 (B)$ &  $\eta_{ie}$  &  $\eta_{rel}$ & $\rho $
\\
\hline
\multicolumn{10}{c }{ $n= 2^{13} -1 =8,191$} \\ \hline
% 1e  	 & 3e7  & 2e-12  	 & 3e-12 	 &  7e-12   &   9e-12 &  1e 	 &  3e-15    &   4e-15  &   7e2 \\
 1e1  	 & 1e7  & 3e-12  	 & 4e-12   &  2e-12   &   3e-12  &  8e0	 &  3e-15    &   4e-15  &   3e3 \\
 1e2  	 & 2e6  & 3e-11  	 & 4e-11   &  3e-13   &   3e-13  &  2e2	 	 &  4e-15    &   5e-15  &   4e3 \\
 1e3  	 & 2e5  & 3e-12  	 & 4e-12  &  3e-14   &   4e-14 	   &  7e3	 &  7e-15    &   9e-15  &   4e3 \\
  1e4  	 & 2e4  & 2e-17  	 & 3e-17    &  8e-15  	 &  1e-14   &   1e5 	 &  7e-15    &   9e-15  &   4e3 \\
 1e5  	 & 5e3  & 5e-16  	 & 6e-16    &  2e-14  	 &  3e-14   &   1e6 	 &  2e-14    &   3e-14  &   4e3 \\
 1e6  	 & 5e3  & 1e-15  	 & 2e-15    &  6e-13  	 &  8e-13   &   4e6 	 &  6e-13    &   8e-13  &   4e3 \\
  \hline
\multicolumn{10}{c }{ $n= 2^{19} -1 =524,287$} \\ \hline
% 1e  	 & 1e11  & 1e-9  	 & 3e-7  &  2e-9   &   5e-7  & - 	 	 &  1e-16    &   3e-14  &   3e2 \\
 1e1  	 & 6e10  & 4e-10  	 & 5e-8 	 &  1e-9   &   2e-7  & - 	 	 &  2e-16    &   2e-14  &   2e3 \\
 1e2  	 & 7e9  & 7e-9  	 & 1e-7 	 &  8e-10   &   2e-8 	 & - 	  &  8e-15    &   2e-13  &   2e4 \\
 1e3  	 & 7e8  & 8e-9  	 & 2e-8 	 &  7e-10   &   2e-9 	 & - 	  &  1e-13    &   4e-13  &   1e5 \\
 1e4  	 & 7e7  & 2e-9  	 & 2e-9    &  1e-10  	 &  2e-10   &   - 	 &  5e-14    &   6e-14  &   3e5 \\
 1e5  	 & 7e6  & 6e-11  	 & 8e-11    &  1e-11  	 &  2e-11   &   - 	 &  6e-14    &   8e-14  &   3e5 \\
 1e6  	 & 7e5  & 1e-18  	 & 2e-18    &  1e-12  	 &  2e-12   &  -	 &  6e-14    &   8e-14  &   3e5 \\
% 1e0  	 & 3e7  & 2e-12  	 & 3e-12    &  1e0  	 &  7e-12   &   9e-12 	 &  3e-15    &   4e-15  &   7e2 \\
% 1e1  	 & 1e7  & 3e-12  	 & 4e-12    &  8e0  	 &  2e-12   &   3e-12 	 &  3e-15    &   4e-15  &   3e3 \\
% 1e2  	 & 2e6  & 3e-11  	 & 4e-11    &  2e2  	 &  3e-13   &   3e-13 	 &  4e-15    &   5e-15  &   4e3 \\
% 1e3  	 & 2e5  & 3e-12  	 & 4e-12    &  7e3  	 &  3e-14   &   4e-14 	 &  7e-15    &   9e-15  &   4e3 \\
\end{tabular}
\label{tbl:ex1a}
\end{center}
\end{table}

We observe that in all cases, the accurate preconditioning produces an inverse-equivalent accuracy $\eta_{ie} $  roughly in the order of machine precision, regardless of the condition number $\kappa_2 (A_n)$. Taking into consideration the results of Example 2 below, $\eta_{ie}$ appears to be proportional to $(1+\rho)\bu$ as indicated by the theory. For the first case where $\kappa_2 (B)$ is computed, $\eta_{ie} $  increases slightly with $\kappa_2 (B)$ but this effect appears to emerges only when $\kappa_2 (B) \ge 10^5$.  With $\eta_{ie} $ in the order of machine precision, the relative error  $\eta_{rel}$ is improved accordingly, which, in this case, is near the machine precision.
%while the accuracy slightly deteriorates for the case $\gamma=10^3$ in the second test when the quality of the preconditioner has deteriorated and the factor $\rho$ is also larger.
In contrast, the solutions by
${\tt A\backslash b}$ and by the Cholesky preconditioning have  relative errors $\eta_{rel}$ of order $\kappa_2 (A) \bu$ as expected, which determines a corresponding $\eta_{ie} $. With larger $\gamma$, $A_n$ becomes less ill-conditioned and the accuracy attained by ${\tt A\backslash b}$ increases. When $\gamma \ge 10^4$ (the first $n$ case) or $\gamma=10^6$ (the second $n$ case), it becomes more accurate than the one by the accurate preconditioning, but since $\kappa_2(B)$ is larger than $\kappa_2 (A_n)$ in those cases, the preconditioning is obviously not expected to be effective.

The results demonstrate that when $\kappa_2 (B)$ is not very large, the accurate preconditioning indeed produces inverse-equivalent accuracy while the  preconditioning solved by a backward stable algorithm  does not improve the solution accuracy at all.

\emph{Example 2.}
Let $A_n = (n+1)^4 T_n^2 + \gamma S_n$, where $T_n$ is as in Example 1,
%the $n\times n$ tridiagonal matrix with diagonals being $2$ and off-diagonals being $-1$,
$S_n$ is a  random sparse integer matrix constructed  using
$
{\tt S_n =floor(10*sprandn(n,n,0.001))} %\;\;{\tt S_n=S_n+S_n'}
$
in MATLAB and $\gamma$ is an integer parameter. Note that $(n+1)^4 T_n^2$ is a finite difference discretization of 1-dimensional biharmonic operator ${d^4 u \over d x^4}$ with the
boundary condition $u= {d^2 u \over  d x^2} =0$ on a uniform mesh on $[0, 1]$ with the meshsize $1/(n+1)$.
%We then consider a general sparse perturbation $S_n$ with varied magnitude of perturbation as parameterized  by
For an integer value of $\gamma$, $A_n$ is an integer matrix and we construct a random integer vector $b$ and the corresponding exact solution $x$ as discussed at the beginning of this section.

If $|\gamma|$ is not too large, preconditioning with $(n+1)^4 T_n^2 $ results in a well-conditioned matrix $B=I+ \frac{\gamma}{(n+1)^4} T_n^{-2} S_n$. We  solve the preconditioned system by GMRES with two way of solving the preconditioning equations: 1. using the Cholesky factorization of $T_n^2$, and 2. using the accurate $LDU$ factorization of $T_n$.  The GMRES is implemented with restart after 50 iterations and the stopping tolerance for relative residual is set as $\sqrt{n} \bu$. Again, we also solve the original system using MATLAB's division operator ${\tt A\backslash b}$.   We compare the computed solutions $\what x$ by the three methods with respect to  $\eta_{ie} $ and   $\eta_{rel}$.

In Table \ref{tbl:ex2}, we present the testing results for $n= 2^{10} -1 =1,023$ and $n= 2^{14} -1 =16,383$. For these two cases respectively, $S_n$ has 1,008 and 257,572 nonzeros with $\|S_n\|_\infty = 75$ and $343$. For each of the $n$ value, we test $\gamma=10, -10^2,10^3, -10^4, 10^5$, $ -10^6, 10^7$.
%For the larger $n$ case, $A_n$ becomes indefinite with all three  negative values of $\gamma$, but for the smaller $n$, it is only indefinite when  $\gamma=-10^2$.
We list in the table $\kappa_2 (A_n)$ and $\kappa_2 (B)$ and $\rho := \|\gamma S_n\|_1 \|x\|_1 /\|b\|_1 $, in addition to $\eta_{ie}$, $\eta_{rel}$.  %(computed using {\tt eigifp}).
%Since the preconditioned matrix $B$ is dense, $\kappa_2 (B)$ is only computed for Table \ref{tbl:ex1a} with {\tt cond}.
%The approximate solutions by all three methods have the relative residuals $\frac{\|b-A_n\what x\|_2}{\|A_n\|_2 \|\what x\|_2}$  approximately  $10^{-17}$; namely all three solutions are backward stable.

\begin{table}[hbt] \small
\begin{center}
\caption{\small {Example 2}: Accuracy for the three methods (${\tt A\backslash b}$, Cholesky Preconditioning, Accurate $LDU$ Preconditioning): $\eta_{ie} = \|\what x- x\|_2/(\|A^{-1}\|_2 \|b\|_2)$, $\eta_{rel} = \|\what x- x\|_2/\|x\|_2$, and $\rho =\|\gamma S_n\|_1 \|x\|_1 /\|b\|_1$.
}
\begin{tabular}{cc|cc|ccc|ccc}
 \multicolumn{2}{c|}{} & \multicolumn{2}{c|}{ ${\tt A\backslash b}$} & \multicolumn{3}{c|}{ Cholesky Precond.} &  \multicolumn{3}{c}{Accurate Precond.} \\
\hline
$\gamma$ & $\kappa_2 (A)$ &  $\eta_{ie} $ &  $\eta_{rel}$
&  $\eta_{ie}$ &  $\eta_{rel}$ & $\kappa_2 (B)$ &  $\eta_{ie}$  &  $\eta_{rel}$ & $\rho$
\\
\hline
\multicolumn{9}{c }{ $n= 2^{10} -1 =1,023$} \\ \hline
 1e1  	 & 2e11  & 3e-10  	 & 1e-7    &  4e-10  	 &  1e-7   &   3e0 	 &  5e-18    &   2e-15  &   2e-2 \\
 -1e2  	 & 2e11  & 8e-10  	 & 3e-7    &  4e-10  	 &  1e-7   &   9e1 	 &  6e-18    &   2e-15  &   2e-1 \\
 1e3  	 & 3e11  & 5e-11  	 & 3e-8    &  4e-10  	 &  2e-7   &   2e4 	 &  4e-18    &   2e-15  &   2e0 \\
 -1e4  	 & 4e10  & 2e-10  	 & 2e-8    &  2e-10  	 &  2e-8   &   2e5 	 &  6e-16    &   6e-14  &   1e1 \\
 1e5  	 & 9e9  & 2e-10  	 & 6e-9    &  2e-10  	 &  5e-9   &   5e6 	 &  7e-14    &   2e-12  &   1e2 \\
 -1e6  	 & 4e9  & 8e-11  	 & 1e-9    &  9e-11  	 &  1e-9   &   2e8 	 &  2e-11    &   3e-10  &   1e3 \\
 1e7  	 & 1e8  & 6e-13  	 & 2e-11    &  1e-10  	 &  4e-9   &   6e8 	 &  2e-11    &   6e-10  &   1e2 \\
% -1e8  	 & 2e7  & 2e-12  	 & 2e-11    &  2e-9  	 &  2e-8   &   1e10 	 &  2e-9    &   2e-8  &   1e3 \\

% 1e  	 & 2e11  & 9e-9  	 & 3e-6      	 &  4e-10   &   1e-7 &  1e	 &  2e-18    &   6e-16  &   1e-3 \\
% -1e  	 & 2e11  & 9e-9  	 & 4e-6    	 &  4e-10   &   2e-7 	&  1e   &  2e-18    &   1e-15  &   1e-3 \\
% 1e1  	 & 1e11  & 8e-9  	 & 2e-6    	 &  3e-10   &   8e-8   &  2e 	 &  5e-18    &   1e-15  &   1e-2 \\
% -1e1  	 & 8e11  & 8e-9  	 & 1e-5     	 &  4e-10   &   7e-7 &  9e 	 &  8e-19    &   1e-15  &   1e-2 \\
% 1e2  	 & 2e10  & 1e-8  	 & 4e-7    	 &  4e-10   &   2e-8 &  2e1  	 &  3e-17    &   1e-15  &   1e-1 \\
% -1e2  	 & 3e10  & 2e-10  	 & 1e-8    	 &  4e-10   &   2e-8 &  2e1  	 &  2e-17    &   1e-15  &   1e-1 \\
\hline
\multicolumn{9}{c }{ $n= 2^{14} -1 =16,383$} \\ \hline
 1e1  	 & 3e16  & 8e-10  	 & 5e-2    &  1e-9  	 &  7e-2   &   5e1 	 &  6e-19    &   3e-11  &   1e-6 \\
 -1e2  	 & 2e15  & 1e-10  	 & 4e-4    &  1e-9  	 &  3e-3   &   2e2 	 &  1e-18    &   3e-12  &   1e-5 \\
 1e3  	 & 6e14  & 2e-10  	 & 3e-4    &  6e-10  	 &  8e-4   &   9e3 	 &  6e-19    &   9e-13  &   1e-4 \\
 -1e4  	 & 5e14  & 5e-11  	 & 7e-5    &  3e-10  	 &  5e-4   &   8e5 	 &  5e-19    &   8e-13  &   9e-4 \\
 1e5  	 & 8e13  & 1e-10  	 & 3e-5    &  3e-10  	 &  7e-5   &   1e7 	 &  9e-17    &   2e-11  &   9e-3 \\
 -1e6  	 & 4e13  & 9e-11  	 & 1e-5    &  1e-10  	 &  2e-5   &   5e8 	 &  2e-15    &   2e-10  &   9e-2 \\
 1e7  	 & 9e12  & 3e-11  	 & 1e-6    &  1e-10  	 &  3e-6   &   1e10 	 &  8e-14    &   2e-9  &   9e-1 \\
% -1e8  	 & 4e11  & 3e-11  	 & 6e-8    &  4e-10  	 &  8e-7   &   6e10 	 &  6e-11    &   1e-7  &   5e0 \\

% 1e0  	 & 5e15  & 1e-8  	 & 1e-1    	 &  1e-9   &   1e-2  &  2e0 	 &  8e-19    &   8e-12  &   1e-7 \\
% -1e0  	 & 4e16  & 3e-11  	 & 3e-3    	 &  1e-9   &   1e-1 &  6e0  	 &  4e-19    &   4e-11  &   1e-7 \\
% 1e1  	 & 9e14  & 9e-9  	 & 2e-2     &  1e-9   &   2e-3 	&  2e1  	 &  5e-19    &   9e-13  &   1e-6 \\
% -1e1  	 & 1e15  & 2e-10  	 & 4e-4    &  1e-9   &   2e-3  &  2e1  		 &  2e-19    &   5e-13  &   1e-6 \\
% 1e2  	 & 1e14  & 1e-8  	 & 2e-3    &  4e-9   &   8e-4 	 &  3e2  	 &  6e-18    &   1e-12  &   1e-5 \\
% -1e2  	 & 7e14  & 1e-10  	 & 2e-4     &  5e-10   &   8e-4 &  3e2 	 &  5e-19    &   8e-13  &   1e-5
\end{tabular}
\label{tbl:ex2}
\end{center}
\end{table}

We observe that the accurate preconditioning produces an inverse-equivalent accuracy $\eta_{ie} $   in the order of machine precision, except when $|\gamma|$ is very large. Comparing with Example 1, $\eta_{ie}$ is about 3 order of magnitude smaller and this seems to be due to a corresponding decrease in  $1+\rho$. As $|\gamma|$ increases, the quality of preconditioning deteriorates. However, its effect on $\eta_{ie} $ emerges  only when $\kappa_2 (B)\ge 10^5$. From that point on,  $\eta_{ie} $ appears proportional to $\kappa_2 (B) (1+\rho)\bu$ as indicated by our theory.
%With $\eta_{ie} $ in the order of machine precision, the relative error  $\eta_{rel}$ is improved accordingly.
%while the accuracy slightly deteriorates for the case $\gamma=10^3$ in the second test when the quality of the preconditioner has deteriorated and the factor $\rho$ is also larger.
%In contrast, the solutions by
%${\tt A\backslash b}$, and by the Cholesky preconditioning have  relative accuracy of order $\kappa_2 (A) \bu$ as expected. This implies very little relative accuracy for the extremely ill-conditioning second case ($\gamma=10$).
Overall, similar behavior as in Example 1 is observed  for this random sparse matrix.

The results of these two examples are in  agreement with our error analysis (Theorem \ref{thm:iter}). The inverse-equivalent accuracy error $\eta_{ie}$ appears proportional to $\kappa_2 (B)(1+\rho)\bu$ although its  dependence on $\kappa_2 (B)$ may appear only when $\kappa_2 (B)$ is quite large. Indeed, its capability to produce an inverse-equivalent accuracy with large $\kappa_2(B)$  is rather surprising. This would allow  a broader application of the accurate preconditioning method than what our theory might suggest.

In the next two examples, we compute the smallest eigenvalue (in absolute value) of $A$  accurately by computing the corresponding largest eigenvalue of $A^{-1}$. We have used both   the Lanczos algorithm with full reorthogonalization and the power method (i.e. the  inverse iteration for $A$) and found the results to be similar. Below, we report the results obtained by the inverse iteration only. In applying $A^{-1}$ at each step of iteration, we solve $Au=v$ by  a preconditioned iterative method. We test solving the preconditioner  by the usual Cholesky factorization or by the  accurate $LDU$ factorization (\cite[Algorithm 1]{ye08}).  With the two ways of solving the preconditioning equations, we compare the final approximate eigenvalues obtained.

\emph{Example 3.}
Consider the eigenvalue problem for the same 1-dimensional convection-diffusion operator as in Example 1:
$
-  u''(x) -   u'(x)    = \lambda u(x)$
\mbox{ on } $(0, \gamma)
$
with $u(0)=u(\gamma)=0$. The eigenvalues of this operator are exactly known \cite[Theorem 1]{retr}:
\[
\lambda_i =  \frac{1}{4} + \frac{\pi^2 i^2}{\gamma^2},\;\;\mbox{ for }\;\;i=1,2, \cdots
\]
Discretizing on a mesh of size $h =\gamma/(n+1)$ as in Example 1, we obtain the same matrix
$A_n =\frac{1}{h^2} T_n - \frac{1}{2h} K_n$.
%$A_n = \left( \frac{n+1}{\gamma}\right)^2 T_n + \frac{1}{2\gamma}(n+1) K_n$.

We approximate $\lambda_1 = \frac{1}{4} + \frac{\pi^2}{\gamma^2}$ by computing the smallest eigenvalue  of $A_n$ using the inverse iteration. At each iteration, we  solve $A_n u=v$ by  the GMRES method as preconditioned by $\frac{1}{h^2} T_n$ with two ways of solving the preconditioner $T_n$: 1. using the Cholesky factorization of $T_n$, and 2. the accurate $LDU$ factorization of $T_n$. We denote the computed smallest eigenvalues by $\mu_1^{chol}$ and $\mu_1^{aldu}$ respectively. The GMRES is implemented with restart after 50 iterations and the stopping tolerance for relative residual is set at $\sqrt{n} \bu$. The stopping tolerance for the eigenvalue-eigenvector residuals of the inverse iteration is also set at $\sqrt{n} \bu$. We use this very  stringent criterion to ensure as accurate results as possible. In all cases, the inverse iteration terminates with the residual satisfying the criterion.

In Table \ref{tbl:ex3}, we present the testing results for $h= 2^{-6}, 2^{-8}, \cdots, 2^{-24}$ and $\gamma=1$. We list the  computed eigenvalues $\mu_1^{chol}$ and $\mu_1^{aldu}$ and their relative errors. We observe that $\mu_1^{chol}$ initially converge quadratically as $h$. However, as $h$ decreases, the matrix becomes increasingly ill-conditioned and the roundoff errors associated with the standard Cholesky preconditioning increase and will dominate the discretization errors at some point (see \cite{ye17}). In this example, this occurs at $h\approx 1.5e-5$, after which further decreasing $h$  actually increases the error for $\mu_1^{chol}$. On the other hand,  the error for  $\mu_1^{aldu}$ decreases quadratically to the order of machine precision. Thus, the accurate preconditioning allows us to compute the smallest eigenvalue of the convection-diffusion operator, whose discretization is nonsymmetric and not diagonally dominant, to the full accuracy of the discretization, up to the machine precision.

\begin{table}[hbt] \small
\begin{center}
\caption{\small {Example 3}: approximation of
$\lambda_1 = \frac{1}{4} + \pi^2 =  10.11960440108936$ (
$\mu_1^{chol}$ - computed eigenvalue by Cholesky preconditioner;
$\mu_1^{aldu}$ - computed eigenvalue by accurate $LDU$ preconditioner.)
}
\begin{tabular}{c|c|c|c|c}
$h$ & %$\lambda_{1,h} $ &
$\mu_1^{chol}$ & $\frac{|\lambda_{1}-\mu_1^{chol}|}{\lambda_{1}}$
& $\mu_1^{aldu}$ & $\frac{|\lambda_{1}-\mu_1^{aldu}|}{\lambda_{1}}$ \\
%$h$ & %$\lambda_{1,h} $ &
%$\mu_1^{chol}$ & $\frac{|\lambda_{1}-\mu_1^{chol}|}{\lambda_{1}}$
%& $\mu_1^{aldu}$ & $\frac{|\lambda_{1}-\mu_1^{aldu}|}{\lambda_{1}}$ \\
\hline
% 6.3e-2  	 & 	  10.08319264527718  	 & 	  3.6e-3  	 & 	 10.08319264527718  	 & 	  3.6e-3 \\
 1.6e-2  	 & 	  10.11732544149765  	 & 	  2.3e-4  	 & 	 10.11732544149762  	 & 	  2.3e-4 \\
 3.9e-3  	 & 	  10.11946195350748  	 & 	  1.4e-5  	 & 	 10.11946195350759  	 & 	  1.4e-5 \\
 9.8e-4  	 & 	  10.11959549807000  	 & 	  8.8e-7  	 & 	 10.11959549806623  	 & 	  8.8e-7 \\
 2.4e-4  	 & 	  10.11960384467139  	 & 	  5.5e-8  	 & 	 10.11960384465017  	 & 	  5.5e-8 \\
 6.1e-5  	 & 	  10.11960436740018  	 & 	  3.3e-9  	 & 	 10.11960436631197  	 & 	  3.4e-9 \\
 1.5e-5  	 & 	  10.11960440025146  	 & 	  8.3e-11  	 & 	 10.11960439891543  	 & 	  2.1e-10 \\
 3.8e-6  	 & 	  10.11960357476229  	 & 	  8.2e-8  	 & 	 10.11960440095356  	 & 	  1.3e-11 \\
 9.5e-7  	 & 	  10.11959786966499  	 & 	  6.5e-7  	 & 	 10.11960440107954  	 & 	  9.7e-13 \\
 2.4e-7  	 & 	  10.11960179526253  	 & 	  2.6e-7  	 & 	 10.11960440108836  	 & 	  9.9e-14 \\
 6.0e-8  	 & 	  10.11996930306172  	 & 	  3.6e-5  	 & 	 10.11960440108905  	 & 	  3.0e-14 \\
% 1.5e-8  	 & 	  10.18751580079536  	 & 	  6.7e-3  	 & 	 10.11960440113352  	 & 	  4.4e-12 \\
\end{tabular}

\label{tbl:ex3}
\end{center}
\end{table}

\smallskip

\emph{Example 4:}
Consider computing the smallest eigenvalue of the 1-dimensional biharmonic problem:
$ {d^4 v \over dx^4} + \rho v =\lambda v $ on $[0, 1]$
with the natural boundary condition $v(0)=v''(0)=v(1)=v''(1)=0$.  Discretizing on a uniform mesh of size $h=1/(n+1)$ leads to  $A_n =\frac{1}{h^4}  T_{n}^2 + \rho I.$
where $T_n$ is the discretization of 1-dimensional Laplacian defined in Example 1.
The eigenvalues of $A_n $ are  $\lambda_{j,h} =\frac{1}{h^4} 16 \sin^4 (j \pi h/2) + \rho$ (see \cite[Lemma 6.1]{demmel97}). We consider $n=2^{16}-1=65,535$ for this example and  $\rho=\pm1,  \pm 10, \pm 10^2, \pm 10^3$. This results in an extremely ill-conditioned  $A_n$ with $\kappa_2 (A_n) \approx 10^{18}$ except in the case of $\rho = -10^2$ when $\kappa_2 (A_n) \approx 10^{20}$.  $A_n$ also becomes indefinite  when $\rho =-10^2$ or $ -10^3$.

We compute the smallest eigenvalue in absolute value, denoted by $\lambda_{\rm absmin}$, of $A_n$  by applying the inverse iteration to $A_n$. Note that this eigenvalue may not be $\lambda_{1,h}$ if $A_n$ is indefinite. In carrying out the inverse iterations, we  solve $A_n x=b$  by the CG (or MINRES if $\gamma<0$) method as preconditioned by $\frac{1}{h^4} T_n^2$ with two way of solving the preconditioner $T_n^2$: 1. using the Cholesky factorization of $T_n^2$, and 2. using accurate $LDU$ factorization of $T_n$. We denote the computed smallest eigenvalues in absolute value by $\mu_1^{chol}$ and $\mu_1^{aldu}$ respectively. The stopping tolerance for relative residual of CG or MINRES is set at $\sqrt{n} \bu$. The stopping tolerance for the eigenvalue-eigenvector residuals of the inverse iteration is also set at $\sqrt{n} \bu$. In our tests, the inverse iteration with the accurate $LDU$ factorization preconditioning produces a  residual satisfying the stopping criterion in all cases. The one with  the Cholesky factorization preconditioning, however, results in stagnating residuals mostly  around $10^{-11}$ that is slightly above the threshold. The latter can be attributed to the inaccuracy in the operator $A_n^{-1}$.

In Table \ref{tbl:ex4}, we present, for each case of $\rho$, the exact eigenvalue $\lambda_{\rm absmin}$,  the  computed eigenvalues $\mu_1^{chol}$ and $\mu_1^{aldu}$ and their relative errors. For all the cases of $\rho$ here, the preconditioned matrix $B = I+\rho h^2 T_n^{-2}$ is well conditioned with $\kappa (B) $ raging between $1$ and $40$. As a result, the accurate preconditioning produces $\mu_1^{aldu}$ that is accurate to the machine precision in all cases. The eigenvalues computed using the preconditioning with the Cholesky factorization $\mu_1^{chol}$ have in most cases  one digit of accuracy. In the case  $\rho =-10^2$, it has even the sign wrong. Again we see that  the accurate preconditioning accurately computes the smallest eigenvalue of this extremely ill-conditioned matrix, even when the matrix is indefinite.

\begin{table}[hbt] \small
\begin{center}
\caption{\small {Example 4}: approximation of the smallest eigenvalue in absolute value
$\lambda_{\rm absmin}$ (
$\mu_1^{chol}$ and $\mu_1^{aldu}$ - computed eigenvalue by Cholesky preconditioning and by accurate $LDU$ preconditioner respectively.
$e^{chol} :=\frac{|\lambda_{\rm absmin}-\mu_1^{chol}|}{|\lambda_{\rm absmin}|}$;
$e^{aldu} :=\frac{|\lambda_{\rm absmin}-\mu_1^{aldu}|}{|\lambda_{\rm absmin}|}$)
}
\begin{tabular}{c|c|c|c|c|c|c|c}
$\rho$ &  $\lambda_{\rm absmin}$ &
$\mu_1^{chol}$ & $e^{chol}$
& $\mu_1^{aldu}$ & $e^{aldu}$ \\
\hline
 1e0   &  98.409090996696  & 	 107.104718485058  &  9e-2  	 & 	 98.409090996693  	 & 	  3e-14 \\
 -1e0  &  96.409090996696  & 	 105.104718492797  &  9e-2  	 & 	 96.409090996693  	 & 	  3e-14\\
 1e1   &  107.409090996696  & 	 116.104718499209  &  8e-2  	 & 	 107.409090996693  	 & 	  3e-14\\
 -1e1  &  87.409090996696  & 	 96.104718508722  &  1e-1  	 & 	 87.409090996693  	 & 	  3e-14\\
 1e2   &  197.409090996696 &  	 206.104718499414  &  4e-2  	 & 	 197.409090996691  	 & 	  3e-14\\
 -1e2   &  -2.590909003304  & 	 6.104718530850  &  3e0  	 & 	 -2.590909003309  	 & 	  2e-12\\
 1e3  &  1097.40909099669  & 	 1106.10471864325  &  8e-3  	 & 	 1097.40909099668  	 & 	  1e-14\\
 -1e3  &  558.545454156402  & 	 716.309982554411  &  3e-1  	 & 	 558.545454156396  	 & 	  1e-14\\

%$\rho$ & $\kappa (A) $ & $\kappa (B) $ & $\lambda_{\rm absmin}$ &
%$\mu_1^{chol}$ & $\frac{|\lambda_{\rm absmin}-\mu_1^{chol}|}{|\lambda_{\rm absmin}|}$
%& $\mu_1^{aldu}$ & $\frac{|\lambda_{\rm absmin}-\mu_1^{aldu}|}{|\lambda_{\rm absmin}|}$ \\
%\hline
% 1e0  & 3e18 & 1e0 &  98.409090996696  	 107.104718485058  &  9e-2  	 & 	 98.409090996693  	 & 	  3e-14 \\
% -1e0  & 3e18 & 1e0 &  96.409090996696  	 105.104718492797  &  9e-2  	 & 	 96.409090996693  	 & 	  3e-14\\
% 1e1  & 3e18 & 1e0 &  107.409090996696  	 116.104718499209  &  8e-2  	 & 	 107.409090996693  	 & 	  3e-14\\
% -1e1  & 3e18 & 1e0 &  87.409090996696  	 96.104718508722  &  1e-1  	 & 	 87.409090996693  	 & 	  3e-14\\
% 1e2  & 1e18 & 2e0 &  197.409090996696  	 206.104718499414  &  4e-2  	 & 	 197.409090996691  	 & 	  3e-14\\
% -1e2  & 1e20 & 4e1 &  -2.590909003304  	 6.104718530850  &  3e0  	 & 	 -2.590909003309  	 & 	  2e-12\\
% 1e3  & 3e17 & 1e1 &  1097.409090996696  	 1106.104718643258  &  8e-3  	 & 	 1097.409090996682  	 & 	  1e-14\\
% -1e3  & 5e17 & 3e1 &  558.545454156402  	 716.309982554411  &  3e-1  	 & 	 558.545454156396  	 & 	  1e-14\\
\end{tabular}
\label{tbl:ex4}
\end{center}
\end{table}

\section{Concluding Remarks}
\label{sec:conclusion}
We have presented an accurate preconditioning method  to solve linear systems with inverse-equivalent accuracy. An error analysis is developed to demonstrate the accuracy that may be achieved by this approach. Numerical examples confirm the analysis but also show that the method works even when the quality of preconditioner is rather low. As an application, we use it to accurately compute the smallest eigenvalue of some differential operator discretizations that are indefinite or nonsymmetric.

For future works, it will be interesting to study a related perturbation theory  and to investigate what appears to be a very mild dependence of the accuracy on the condition number of the preconditioned matrix.  It will also be interesting to study whether our method can be used with preconditioners that are defined through their inverses, such as multilevel preconditioners \cite{xu} and sparse approximate inverse preconditioners \cite{BenziCullumTuma:SAINV,BenziTuma:RIF}.
%in the contexts of various applications such as discretizations of differential operators.

\bigskip

{\bf Acknowledgement}: I would like to thank Prof. Jinchao Xu for some interesting discussions on multilevel preconditioners that have inspired this work. I would also like to thank Kasey Bray for many helpful comments on a draft of this paper.

%
%\bibliography{ref}{}
%\bibliographystyle{plain}
%\end{document}

\end{document}